\documentclass[12pt]{article}
\usepackage[paperwidth=230mm, paperheight=310mm]{geometry}
\usepackage{graphicx}
\usepackage{amssymb}
\usepackage{amsmath}
\usepackage{hyperref}
\usepackage{enumerate}
\usepackage{amsthm}
\usepackage{amsmath}
\usepackage{float}
\usepackage{mathtools}
\usepackage{xcolor}
\definecolor{lanse}{RGB}{0,112,192}
\definecolor{zise}{RGB}{112,48,160} 
 \definecolor{hongse}{RGB}{200,0,0} 
\usepackage{lineno}
\makeatletter
\renewenvironment{proof}[1][\proofname]{%
   \par\pushQED{\qed}\normalfont%
   \topsep6\p@\@plus6\p@\relax
   \trivlist\item[\hskip\labelsep\bfseries#1\@addpunct{.}]%
   \ignorespaces
}{%
   \popQED\endtrivlist\@endpefalse
}
\makeatother

\numberwithin{equation}{section}

\newtheorem{theorem}{Theorem}
\newtheorem{proposition}[theorem]{Proposition}
 \numberwithin{theorem}{section}
 
\newtheorem{lemma}[theorem]{Lemma}
\newtheorem{mydef}[theorem]{Definition}
\newtheorem{remark}[theorem]{Remark}

\makeatletter
\def\keywords{\xdef\@thefnmark{}\@footnotetext}
\makeatother

\renewcommand{\P}{\mathbb{P}}
\newcommand{\E}{\mathbb{E}}
\newcommand{\R}{\mathbb{R}}

\newcommand{\Z}{\mathbb{Z}}

\newcommand{\N}{\mathbb{N}}

\newcommand{\cA}{\mathcal A}
\newcommand{\cI}{\mathcal I}
\newcommand{\cC}{\mathcal C}

\newcommand{\cN}{\mathcal{N}}

\newcommand{\cH}{\mathcal{H}}

\newcommand{\cF}{\mathcal F}

\newcommand{\cU}{\mathcal U}

\newcommand{\cG}{\mathcal G}

\newcommand{\eps}{\varepsilon}

 \newcommand{\nn}{\nonumber}
 \newcommand{\no}{\noindent}

\allowdisplaybreaks

\begin{document}
\keywords{\today}%
\keywords{AMS 2020 \emph{subject classification.} Primary: 60K35, 60J80, 60J68}%
\keywords{\emph{Key words and phrases.} Bernoulli bond percolation, critical probability, SIR epidemic, branching random walk}%

\author{
Jieliang Hong 
}
\title{A lower bound for $p_c$ in range-$R$ bond percolation in four, five and six dimensions}

\date{{\small  {\it  Department of Mathematics, Southern University of Science and Technology,\\
 Shenzhen, China\\
 E-mail:  {\tt hongjl@sustech.edu.cn} \\
   }
  }
  }

\maketitle
\begin{abstract}
For the range-$R$ bond percolation in $d=4,5,6$, we obtain a lower bound for the critical probability $p_c$ for $R$ large, agreeing with the conjectured asymptotics and thus complementing the corresponding results of Van der Hofstad-Sakai \cite{HS05} for $d>6$, and Frei-Perkins \cite{FP16}, Hong \cite{Hong21} for $d\leq 3$. The proof follows by showing the extinction of the associated SIR epidemic model and introducing a self-avoiding branching random walk where births onto visited sites are suppressed and the total range of which dominates that of the SIR epidemic process. 
\end{abstract}
\section{Introduction and the main result} \label{4s1}

Set $R\in \N$. The range-$R$ bond percolation takes place on the scaled integer lattice $\Z_R^d=\Z^d/R=\{x/R: x\in \Z^d\}$, which is equivalent to the Bernoulli bond percolation on $\Z^d$ where bonds are allowed to form over a long range when $R$ is large. Such range-$R$ bond percolation dates back at least to the ``spread-out'' model considered in Hara and Slade \cite{HS90}. It can be used to model the spread of the disease in a large population when the range of infection can be very long, in particular, due to the increased interactions and the more frequent communications within the population. Let $x, y\in \Z^d_R$ be neighbours if $0<\|x-y\|_\infty\leq 1$ where $\|\cdot \|_\infty$ denotes the $l^\infty$ norm on $\R^d$, and write $x\sim y$ if $x,y$ are neighbours. Let $\cN(x)$ be the set of neighbours of $x$ and denote its size by 
\[V(R):=|\cN(x)|=|\{y\in \Z^d_R: 0<\|y-x\|_\infty\leq 1\}|=(2R+1)^d-1.\] Here $|S|$ stands for the cardinality of a finite set $S$. Now as usual in Bernoulli bond percolation, we include the edge $(x,y)$ for any two neighbors $x\sim y$ with some probability $p>0$, independently of all other edges. Denote by $G=G_R$ the resulting subgraph with vertex set $\Z_R^d$ and edge set is the set of open edges. Define $\cC(0)$ to be the cluster in $G$ containing $0$. The critical probability $p_c$ is then given by
\[
p_c=p_c(R)=\inf\{p: \P_p(|\cC_0|=\infty)>0\}.
\]
We are interested in finding the asymptotic behavior of $p_c(R)$ as $R\to \infty$. \\

Write $f(R)\sim g(R)$ as $R\to \infty$ if $\lim_{R\to \infty} f(R)/g(R)=1$.  It was first shown in Penrose \cite{Pen93} that
\begin{align}\label{4e1.449}
p_c(R) \sim \frac{1}{V(R)} \text{ as } R\to \infty, \text{ in } d\geq 2,
\end{align}
which is analogous to the results of Kesten \cite{Kes90} for the nearest neighbour percolation on $\Z^d$ when $d\to \infty$. Later in higher dimensions $d> 6$, Van der Hofstad and Sakai \cite{HS05} use lace expansion to get finer asymptotics on $p_c(R)$:
\begin{align}\label{4e10.10}
p_c(R)V(R)-1 \sim \frac{\theta_d}{R^d},
\end{align}
where $\theta_d$ is given in terms of a probability concerning random walk with uniform steps on $[-1,1]^d$. See \eqref{4c10.64} below for the explicit expression of $\theta_d$.
The extension of \eqref{4e10.10} to $d=6,5$ has been conjectured by the two authors in \cite{HS05} while in dimension $d=4$, it has been conjectured by Edwin Perkins and Xinghua Zheng [private communication] that
\begin{align}\label{4ec10.10}
p_c(R)V(R)-1 \sim \frac{\theta_4 \log R}{R^4} \text{ in } d=4.
\end{align}
They also conjecture the constant $\theta_4$ to be $9/(2\pi^2)$, agreeing with our result below. In lower dimensions $d=2,3$, Frei-Perkins \cite{FP16} and Hong \cite{Hong21} give respectively the lower and upper bound for $p_c$ that suggests the correct asymptotics for $p_c(R)V(R)-1$ should be 
\begin{align}\label{4e10.11}
p_c(R)V(R)-1\sim \frac{\theta_d}{R^{d-1}} \text{ in } d=3,2
\end{align}
for some constant $\theta_d>0$ that depends on the dimension. They use in particular the ideas from branching random walk (BRW) or superprocess to study the range-$R$ bond percolation. Moreover, the fact that the local time of super-Brownian motion exists when $d\leq 3$ is essential in \cite{Hong21}, allowing the author to apply the theory of superprocess to study the asymptotics of $p_c(R)$. However, the local time does not exist in $d\geq 4$, so new tools will be needed. \\
 
In this paper, we adapt the methods from Frei-Perkins \cite{FP16} for the SIR epidemic process and the ideas from Durrett-Perkins \cite{DP99} for the contact process to study percolation in the intermediate dimensions $4\leq d\leq 6$, and obtain the lower bound for $p_c(R)$ that matches the conjectured asymptotics as in \eqref{4e10.10} for $d=5,6$ and \eqref{4ec10.10} for $d=4$.

Let $Y_1,Y_2, \cdots$ be i.i.d. random variables on $\R^d$ so that $Y_1$ is uniform on $[-1,1]^d$. Set $U_n=Y_1+\cdots+Y_n$ for $n\geq 1$. When $d\geq 5$, define
\begin{align}\label{9ec10.64}
b_d=2^{-d}  \sum_{n=1}^\infty \sum_{k=n}^{2n}   \P(U_{k}\in [-1,1]^d)=2^{-d}  \sum_{k=1}^\infty \Big[\frac{k+2}{2}\Big]  \P(U_{k}\in [-1,1]^d),
\end{align}
where $[x]$ is the largest integer that is less than or equal to $x$. In $d=4$, we let $b_4={9}/(2\pi^2)$. \\

\begin{theorem}\label{4t0}
Let $4\leq d\leq 6$. For any $\theta<b_d$, there exists some constant $c(\theta)>0$ so that for any positive integer $R>c(\theta)$, we have
\begin{align} 
p_c(R)V(R)-1\geq 
\begin{dcases}
 \frac{\theta \log R}{R^4}, &\quad \text{ in } d=4;\\
  \frac{\theta}{R^{d}}, &\quad \text{ in } d=5,6.
\end{dcases}
\end{align}
In particular, the above implies
 \begin{align} \label{4c1.0}
\liminf_{R\to \infty} \frac{\big[p_c(R)V(R)-1\big]R^4}{\log R} \geq b_4, \quad \text{ in } d=4,
\end{align}
and
 \begin{align} \label{4c1.1}
\liminf_{R\to \infty} \big[p_c(R)V(R)-1\big] R^d  \geq b_d, \quad \text{ in } d=5,6.
\end{align}
\end{theorem}
\begin{remark}
\no (a) As will be justified later in Section \ref{4s1.3}, our methods for proving the lower bound are believed to be sharp for $d=4$, so we conjecture that the limit as in \eqref{4c1.0} exists, and equals $b_4$, i.e.
 \begin{align} \label{4e1.70}
\lim_{R\to \infty} \frac{\big[p_c(R)V(R)-1\big]R^4}{\log R}=b_4, \quad \text{ in } d=4.
\end{align}
We were informed by Edwin Perkins that the constant $b_4={9}/(2\pi^2)$ had been conjectured by him and Xinghua Zheng. \\

\no (b) In $d=5,6$, Van der Hofstad-Sakai \cite{HS05} conjecture that the limit as in \eqref{4c1.1} exists and equals (see (1.18) of \cite{HS05} with $U^{\star (n+1)}=2^{-d} \P(U_n\in [-1,1]^d)$)
\begin{align}\label{4c10.64}
\widetilde{b_d}=2^{-d}  \P(U_{1}\in [-1,1]^d)+2^{-d}  \sum_{k=2}^\infty \frac{k+2}{2}  \P(U_{k}\in [-1,1]^d).
\end{align}
One can easily check 
\begin{align} \label{4c10.67}
\widetilde{b_d}- b_d=2^{-d}  \sum_{k=1}^\infty \frac{1}{2}  \P(U_{2k+1}\in [-1,1]^d)>0,
\end{align}
hence our results partially confirm their conjectures. It would be desirable to upgrade the lower bound from $b_d$ to $\widetilde{b_d}$, the gap between which we think comes from, say, using the language of branching random walk, the collisions when two or more particles are sent to the same location at the same time. See more discussions below in Remark \ref{4r1}.
\end{remark}

From now on, we fix $0<\theta<b_d$ and consider $R\in \N$ with $R\geq 100+4\theta$. Let $N_R>0$ be given by
\begin{align}\label{6e1.3}
N_R=
\begin{cases}
R^d,&d\geq 5\\
R^4/\log R, &d=4.
\end{cases}
\end{align}
Define
\begin{align}\label{6e1.4}
p(R)=\frac{1+\frac{\theta}{N_R}}{V(R)}.
\end{align}
It suffices to show that $p_c(R)\geq p(R)$ for $R$ large, or equivalently, the percolation does not occur if the Bernoulli parameter $p$ is set to be $p(R)$.\\

\no ${\bf Convention\ on\ Functions\ and\ Constants.}$ Constants whose value is unimportant and may change from line to line are denoted $C, c, c_d, c_1,c_2,\dots$. \\

\section*{Acknowledgements}
The author's work was partly supported by Startup Funding XXX. We thank Edwin Perkins for telling us their conjectured constant for $d=4$.

\section{Proof of the lower bound}

 \subsection{SIR epidemic models} \label{4s1.2}
To prove such a lower bound as in Theorem \ref{4t0}, we will use the connection between the bond percolation and the discrete-time SIR epidemic model following \cite{FP16} and \cite{Hong21}. The SIR epidemic process on $\Z^d_R$ is defined by recording the status of all the vertices on $\Z_R^d$:  For any time $n\geq 0$, each vertex $x\in \Z^d_R$ is either infected, susceptible or recovered, the set of which is denoted respectively by $\eta_n, \xi_n, \rho_n$. Given the finite initial configurations of infected sites, $\eta_0$, and recovered sites, $\rho_0$, the epidemic evolves as follows: An infected site $x\in \eta_n$ infects its susceptible neighbor $y\in \xi_n$, $y\sim x$ with probability $p=p(R)$, where the infections are conditionally independent given the current configuration. Infected sites at time $n$ become recovered at time $n+1$, and recovered sites will be immune from further infections and stay recovered. Denote by $(x,y)$ the undirected edge between two neighbors $x\sim y$ in $\Z_R^d$ and let $E(\Z_R^d)$ be the set of all such edges. If we assign i.i.d. Bernoulli random variables $B(e)$ with parameter $p=p(R)$ to each edge $e\in E(\Z_R^d)$, then the above SIR epidemic process can be formulated as
\begin{align}\label{4e10.14}
\eta_{n+1}=&\bigcup_{x\in \eta_n} \{y\in \xi_n: B(x,y)=1\},\quad \rho_{n+1}=\rho_{n}\cup \eta_n,\quad \xi_{n+1}=\xi_{n}\backslash \eta_{n+1}.
\end{align}

To give a more specific description of the above SIR epidemic, we denote by $\cF_n=\sigma(\rho_0, \eta_k, k\leq n)$ the $\sigma$-field generated by the SIR epidemic.  One can easily conclude from \eqref{4e10.14} that $\rho_n \in \cF_n$ since $\rho_n=\rho_0 \cup \eta_1\cup\cdots \cup \eta_{n-1}$ and $\xi_n\in \cF_n$ by $\xi_n=(\eta_n\cup \rho_n)^c$.  Let $\partial C_n$ to be the set of infection edges given by
\begin{align}
\partial C_n:=\{(x,y)\in E(\Z_R^d): x\in \eta_n, y\in \xi_n\}. 
 \end{align}
 Then $\partial C_n \in \cF_n$. For each $y\in \Z_R^d$, define
 \begin{align}\label{4e7.55}
D_n(y)=\{x\in \eta_n: (x,y) \in \partial C_n\}
 \end{align}
 to be the set of infected sites that will possibly infect $y$.
 It follows from (1.6) of \cite{FP16} that if $S=\{(x_i,y_i): i\leq m\}$ is a set of distinct edges in $\Z_R^d$ and $V_0=\{y_i: i\leq m\}$, then for any $V\subset V_0$,
   \begin{align}\label{4e7.01}
\P(\eta_{n+1}=V|\cF_n)=\prod_{y\in V_0-V} (1-p)^{|D_n(y)|} \prod_{y\in V} \Big[1-(1-p)^{|D_n(y)|}\Big] \text{ a.s. on } \{\partial C_n=S\}.
\end{align}
The law of the SIR epidemic can be uniquely determined by \eqref{4e7.01} and the joint law of $(\eta_0,\rho_0)$. We refer the reader to Sections 1.2 and 2.1 of \cite{FP16} for more details. \\

Now that the SIR epidemic has been constructed, we may consider its extinction/survival. 
\begin{mydef}\label{4def1.3}
We say that a SIR epidemic {\bf survives} if with positive probability, $\eta_n\neq \emptyset$ for all $n\geq 1$; we say the epidemic becomes {\bf extinct} if with probability one, $\eta_n= \emptyset$ for some finite $n\geq 1$.
\end{mydef}

{\bf Equivalence between the bond percolation and the SIR epidemic:} The connection between the range-$R$ bond percolation and the SIR epidemic can be described as follows: If the epidemic $\eta$ starting with $\eta_0=\{0\}$ and  $\rho_0=\emptyset$ survives, then with positive probability, there is an infinite sequence of distinct infected sites $\{x_k,k\geq 0\}$ with $x_0=0$ such that $x_k\in \eta_k$, $x_{k}$ is a neighbor of $x_{k-1}$, and $x_{k-1}$ infects $x_k$ at time $k$. Hence the edge between $x_{k-1}$ and $x_k$ is open for all $k\geq 1$. This gives that with positive probability, percolation occurs from $\eta_0=\{0\}$ to infinity in the range-$R$ bond percolation.  Conversely, if percolation from $\{0\}$ to infinity occurs in the percolation model, then an infinite sequence of distinct sites for infection must exist and so the epidemic survives. \\

The above implies that to prove Theorem \ref{4t0}, it suffices to show that the SIR epidemic $\eta$ starting with $\eta_0=\{0\}$ and  $\rho_0=\emptyset$ becomes extinct. Throughout the rest of this paper, we will only consider the epidemic with finite initial condition $(\eta_0, \rho_0)$. For any disjoint finite sets $\eta_0$ and $\rho_0$, one may use \eqref{4e10.14} with an easy induction to conclude both $\eta_n$ and $\rho_n$ are finite for all $n\geq 0$. Hence it follows that
\begin{align}\label{45equiv}
 \cup_{n=0}^\infty \eta_n \text{ is not a compact set} \Leftrightarrow\eta_n\neq \emptyset, \forall n\geq 0.
\end{align}

To make a summary, it remains to show that the SIR epidemic $\eta$ starting with $\eta_0=\{0\}$ and  $\rho_0=\emptyset$ satisfies
\begin{align} \label{4equiv}
 \text{ with probability one } \cup_{n=0}^\infty \eta_n \text{ is a compact set.}
\end{align}
We will do this by coupling the SIR epidemic with an appropriate branching envelope.

\subsection{A modified branching envelope} \label{4s1.3}

First we will couple the epidemic $\eta$ with a dominating branching random walk $Z=(Z_n, n\geq 0)$ on $\Z_R^d$. The state space for our branching random walk in this paper is the space of nonnegative point measures on $\Z_R^d$ denoted by $M_P(\Z_R^d)$: for any $\mu \in M_P(\Z_R^d)$, there are some $n\geq 0$, and $a_k\geq 0$, $x_k\in \Z_R^d$, $\forall 1\leq k\leq n$ such that $\mu=\sum_{k=1}^n a_k \delta_{x_k}$. Write $\mu(x)=\mu(\{x\})$. For any function $\phi: \Z^d_R \to \R$, write $\mu(\phi)=\sum_{x\in \Z^d_R} \phi(x) \mu(x)$. Set $|\mu|=\mu(1)$ to be its total mass.  

Totally order the set $\cN(0)$ as $\{e_1, \cdots, e_{V(R)}\}$ and then totally order each $\cN(0)^n$ lexicographically. Following Section 2.2 of Frei and Perkins \cite{FP16}, we will use the following labeling system for our particle system:
\begin{align}\label{4e1.16}
I=\bigcup_{n=0}^\infty  \cN(0)^n=\{(\alpha_1, \cdots, \alpha_n): \alpha_i \in \cN(0), 1\leq i\leq n\},
\end{align}
where $\cN(0)^0=\{\emptyset\}$ labels the root index. Let $|\emptyset|=0$. If $\alpha=(\alpha_1, \cdots, \alpha_n)$, we let $|\alpha|=n$ be the generation of $\alpha$, and write
$\alpha|i=(\alpha_1, \cdots, \alpha_i)$ for $1\leq i\leq n$. Let $\pi \alpha=(\alpha_1, \cdots, \alpha_{n-1})$ be the parent of $\alpha$ and let $\alpha \vee e_i=(\alpha_1, \cdots, \alpha_n, e_i)$ be an offspring of $\alpha$ whose position relative to its parent is $e_i$. Assign an i.i.d. collection of Bernoulli random variables $\{B^\alpha: \alpha \in I, |\alpha|>0\}$ to the edges connecting the locations of $\alpha$ and its parent $\pi \alpha$ so that the birth in this direction is valid with probability $p(R)$ (and invalid with probability $1-p(R)$).  For each $n\geq 1$, write $\alpha \approx n$ iff $|\alpha|=n$ and $B^{\alpha|i}=1$ for all $1\leq i\leq n$ so that such an $\alpha$ labels a particle alive in generation $n$.  For each $\alpha \in I$, define its current location by 
\begin{align}\label{4e1.17}
Y^\alpha=
\begin{cases}
\sum_{i=1}^{|\alpha|}  \alpha_i, &\text{ if } \alpha\approx |\alpha|,\\
\Delta, &\text{ otherwise. }
\end{cases}
\end{align}
Define for any $n\geq 0$ that
\begin{align}\label{4eb2.21}
Z_n:=\sum_{|\alpha|= n} \delta_{Y^\alpha} 1(Y^\alpha \neq \Delta).
\end{align}
 Then $(Z_n)$ gives the empirical distribution of a branching random walk where in generation $n$, each particle gives birth to one offspring to each of its $V(R)$ neighboring positions independently with probability $p(R)$. 
Define $Z_n(x)=Z_n(\{x\})$ for any $x\in \Z^d_R$. .\\

For $\mu,\nu \in M_P(\Z_R^d)$, we say $\nu$ {\bf dominates} $\mu$ if $\nu(x)\geq \mu(x)$ for all $x\in \Z_R^d$. For any set $Y\subset \Z_R^d$, by slightly abusing the notation, we write $Y:=\sum_{x\in Y} \delta_{x}$ so that the set $Y$ naturally defines a point measure $Y\in M_P(\Z^d_R)$ taking values in $\{0,1\}$. In particular we define $\eta_n\in M_P(\Z^d_R)$ for each $n\geq 0$ by letting $\eta_n:=\sum_{x\in \eta_n} \delta_{x}$. The following lemma is from Proposition 2.3 of \cite{FP16} that defines the coupled SIR epidemic $(\eta_n)$ inductively with the dominating $(Z_n)$.

 \begin{lemma}\label{4l1.4}
On a common probability space, we can define a SIR epidemic process $\eta$ starting from $(\{0\}, \emptyset)$, and a branching random walk ${Z}$ as in \eqref{4eb2.21}, such that 
\[
 \eta_n(x)\leq {Z}_n(x), \forall x\in \Z_R^d, n\geq 0.
\]
\end{lemma}

The above coupling, however, will not give the extinction of the SIR epidemic as the dominating BRW is supercritical and survives with positive probability (recall from \eqref{6e1.4} that $p(R)V(R)=1+\frac{\theta}{N_R}>1$). So we consider another $M_P(\Z_R^d)$-valued process $\{\widetilde{Z}_n, n\geq 0\}$ defined inductively by
 \begin{align}\label{4eb2.24}
\widetilde{Z}_n=\sum_{|\alpha|= n} \delta_{Y^\alpha} 1(Y^\alpha \neq \Delta, Y^\alpha \notin \cup_{k=0}^{n-1}S(\widetilde{Z}_k)).
\end{align}
 In the above, $S(\mu)=\text{Supp}(\mu)$ is the support of measure $\mu\in M_P(\Z_R^d)$. By defining such a process, we obtain the branching random walk where particles never give birth to places that have been colonized before, an analog to the well-known self-avoiding random walk. In what follows, we will call $\widetilde{Z}$ ``{\bf a self-avoiding branching random walk}''. Note we still allow two particles to give birth to the same site at the same time.  By definition, $Z_n$ dominates $\widetilde{Z}_n$ for any $n\geq 0$. Nevertheless, it is not true that $\widetilde{Z}_n$ will dominate the SIR epidemic $\eta_n$ as there might exist some site $x$ that is visited earlier by $\widetilde{Z}$ and later by $\eta$. Fortunately, the total colony of $\widetilde{Z}$ dominates that of $\eta$ by the following lemma, so $\widetilde{Z}$ suffices for our purposes given \eqref{4equiv}.

 \begin{lemma}\label{4l1.5}
On a common probability space, we can define a SIR epidemic process $\eta$ starting from $(\{0\}, \emptyset)$, and a self-avoiding branching random walk $\widetilde{Z}$ as in \eqref{4eb2.24}, such that 
\[
\cup_{k=0}^{n} \eta_k \subset \cup_{k=0}^{n} S(\widetilde{Z}_k),  \forall n\geq 0.
\]
\end{lemma}
\begin{proof}
 Let $\widetilde{Z}$ be as in \eqref{4eb2.24}. We will define the coupled SIR epidemic process $(\eta_n, \rho_n, \xi_n: n\geq 0)$ inductively on $n$ so that
 \begin{align}\label{4e2.50}
 \cup_{k=0}^{j} \eta_k \subset \cup_{k=0}^{j} S(\widetilde{Z}_k), \forall j\leq n
\end{align}
 holds and $(\eta_n, \rho_n, \xi_n: j\leq n)$ has the law of a SIR epidemic process with filtration $\cF_n:=\sigma(\rho_0,\eta_k, k\leq n)$. First let $\eta_0=\{0\}$ and $\rho_0=\emptyset$. Assuming the above holds for some $n\geq 0$, we will prove the case for $n+1$. Let $Y_n=\{y\in \Z_R^d: \exists x\in \eta_n, (x,y)\in \partial C_n\}$ be the set of potential infected individuals at time $n+1$. If $x\in \eta_n$, then 
 \eqref{4e2.50} implies there exists some unique $0\leq k_n^x\leq n$ such that $\widetilde{Z}_{k_n^x} (x)\geq 1$ and the set $K:=\{\alpha \approx k_n^x: Y^\alpha=x\}$ is not empty. The uniqueness of $k_n^x$ follows from the definition of $\widetilde{Z}$. Use the total order of $I$ to pick the minimal element in $K$ denoted by $\alpha_{k_n^x}$ and define
  \begin{align}\label{4e2.51}
&\eta_{n+1} =\{y\in Y_n: \exists x\in D_n(y), B^{\alpha_{k_n^x} \vee (y-x)}=1\},\nn\\
& \xi_{n+1}=\xi_n-\eta_{n+1}, \quad \rho_{n+1}=\xi_n\cup \eta_n,
\end{align}
where $D_n(y) \subset \eta_n$ is as in \eqref{4e7.55} and is $\cF_n$-measurable. Intuitively speaking, if $k_n^x=n$ for $x\in \eta_n$, then we let $x$ infect its neighboring sites as in a usual SIR epidemic; if $k_n^x<n$, we use the historical trajectory of $\{\widetilde{Z}_{k}, k\leq n\}$ to define the sites infected by $x$ at time $n+1$. In any case, one can check that \eqref{4e2.50} holds for $n+1$. Moreover, the infection events $\{B^{\alpha_{k_n^x} \vee (y-x)}=1\}$
are independent of $\cF_n$ as they never appear before in the definition of $\eta_k$ for any $1\leq k\leq n$, meaning that this is the first time for the SIR epidemic to infect those $y\in \eta_{n+1}$. 

To rigorously check the inductive step, we note that for any $y\in \eta_{n+1}$, by definition \eqref{4e2.51} there exists some $x\in D_n(y)\subset \eta_n$ such that $B^{\alpha_{k_n^x} \vee (y-x)}=1$ for some $\alpha_{k_n^x} \approx k_n^x$ with $0\leq k_n^x\leq n$. So we get $\alpha_{k_n^x} \vee (y-x) \approx k_n^x+1$ and 
\[
Y^{\alpha_{k_n^x} \vee (y-x)}=Y^{\alpha_{k_n^x}}+(y-x)=y.
\]
If $y\in  \cup_{m=0}^{k_n^x} S(\widetilde{Z}_m)$, then it is immediate that $y\in \cup_{m=0}^{n+1} S(\widetilde{Z}_m)$; if $y\notin  \cup_{m=0}^{k_n^x} S(\widetilde{Z}_m)$, then by the definition \eqref{4eb2.24} for $\widetilde{Z}_{k_n^x+1}$, we get $y\in S(\widetilde{Z}_{k_n^x+1})$, giving $y\in \cup_{m=0}^{n+1} S(\widetilde{Z}_m)$ as well. In any case, we get $y\in \cup_{m=0}^{n+1} S(\widetilde{Z}_m)$, and so
  \begin{align*} 
 \eta_{n+1} \subset \cup_{k=0}^{n+1} S(\widetilde{Z}_k).
\end{align*}
Together with the inductive assumption \eqref{4e2.50}, we conclude
  \begin{align*} 
 \cup_{k=0}^{j} \eta_k \subset \cup_{k=0}^{j} S(\widetilde{Z}_k), \forall j\leq n+1,
\end{align*}
as required.

It remains to prove that the above-defined process $(\eta_n)$ has the law of a SIR epidemic process, which will follow if one shows that it satisfies \eqref{4e7.01}: If $S=\{(x_i,y_i): 1\leq i\leq m\}$ is a set of distinct edges in $\Z_R^d$ and $V\subset \{y_i: 1\leq i\leq m\}=V_2$, then
  \begin{align}\label{4e2.52}
\P(\eta_{n+1}=V|\cF_n)=\prod_{y\in V_2-V} (1-p)^{|D_n(y)|} \prod_{y\in V} [1-(1-p)^{|D_n(y)|}] \text{ a.s. on } \{\partial C_n=S\}.
\end{align}
To see this, by \eqref{4e2.51}, the left-hand side equals
 \begin{align} 
\P\Big(&\forall y\in V_2-V, \forall x\in D_n(y), B^{\alpha_{k_n^x} \vee (y-x)}=0,\nn\\
&\text{ and } \forall y\in V, \exists x\in D_n(y), B^{\alpha_{k_n^x} \vee (y-x)}=1\Big|\cF_n\Big).
\end{align}
 Then \eqref{4e2.52} follows by the observing that the indexes $\{\alpha_{k_n^x} \vee (y-x)\}$ are distinct by the choice of $k_n^x$, $\alpha_{k_n^x}$ and that conditioning on $\cF_n$, $\{B^{\alpha_{k_n^x} \vee (y-x)}\}$ are i.i.d. Bernoulli random variables.
\end{proof}
  
The above lemma implies that it suffices to show the extinction of $\widetilde{Z}$.

\begin{proposition}\label{4p1}
Let $4\leq d\leq 6$. There exists some constant $C(d)>0$ such that for any $R>C(d)$, there exists some $1\leq k_0\leq [N_R]+1$ such that $\E(|\widetilde{Z}_{k_0}|)\leq 1-(\frac{b_d-\theta}{2}\wedge \frac{1}{4})<1$.
\end{proposition}

Assuming Proposition \ref{4p1}, we will finish the proof of our main result Theorem \ref{4t0}.
\begin{proof}[Proof of Theorem \ref{4t0} assuming Proposition \ref{4p1}]
 Consider a branching random walk $(\gamma_n, n\geq 0)$ on $\Z_R^d$ starting from a single ancestor at the origin. For each $n\geq 1$, we let all the individuals in $\gamma_{n-1}$ give birth to independent copies of $\widetilde{Z}_{k_0}$ to obtain $\gamma_n$, in which an offspring is suppressed if and only if the particles jump onto a site that has been visited before by another offspring of the same ancestor. There is no such ancestral restriction in the original self-avoiding branching random walk $(\widetilde{Z}_n)$, meaning it's more likely that particles in $(\widetilde{Z}_n)$ will collide, so one may couple $(\widetilde{Z}_n)$ with $(\gamma_n)$ so that $\gamma_n$ dominates $\widetilde{Z}_{nk_0}$ for any $n\geq 0$. Proposition \ref{4p1} implies that $\E(|\widetilde{Z}_{k_0}|)$ is strictly less than $1$ for $R$ large.  The classical theory of branching process then tells us that $|\gamma_n|=0$ for $n$ large a.s., giving $|\widetilde{Z}_{nk_0}|=0$ for $n$ large. Hence with probability one, $\cup_{k=0}^\infty S(\widetilde{Z}_k)$ is compact and so is $\cup_{k=0}^\infty \eta_k$ by Lemma \ref{4l1.5}. It follows from \eqref{45equiv} that the SIR epidemic $(\eta_n)$ becomes extinct, so percolation does not occur a.s.
\end{proof}

 \begin{remark}\label{4r1}
With some appropriate initial condition $\widetilde{Z}_0$, we do conjecture that $X_t^{N_R}=\frac{1}{N_R}\widetilde{Z}_{[tN_R]}$ will converge to a super-Brownian motion with drift $\theta-b_d$ for all $d\geq 4$. An ongoing work of the author \cite{Hong24} proves such convergence of the SIR epidemic processes. However, there is still a gap between $\eta_n$ and $\widetilde{Z}_n$: The difference between $\eta_n$ and $\widetilde{Z}_n$ comes from the events when two or more particles attempting to give birth to the same site at the same time. If we let $\Gamma_{n}(x)$ denote the number of the failed infections at site $x$ and time $n$ in the SIR epidemic $(\eta_n)$, meaning that if $k$ infected individuals simultaneously attempt to infect $x$ at time $n$, then $\Gamma_{n}(x)=(k-1) \vee 0$. One can show as in the proof of Lemma 8.1 in \cite{Hong21} (see also Lemma 9 of \cite{LZ10}) that for a branching random walk starting from a single ancestor at the origin, we have
  \begin{align} 
\E \Big(\sum_{n=1}^{N_R}  \sum_{x\in \Z^d_R}\Gamma_{n}(x)\Big)=o(1),  \text{ in } d=4,
\end{align}
and 
  \begin{align}\label{4e5.66}
\E \Big(\sum_{n=1}^{N_R}  \sum_{x\in \Z^d_R}\Gamma_{n}(x)\Big)=O(1), \text{ in } d\geq 5.
\end{align}
So one would expect that our lower bound for $p_c$ should be sharp in $d=4$ while in $d\geq 5$, the interesting gap between $b_d$ and $\widetilde{b_d}$ appears due to \eqref{4e5.66}.
\end{remark}

It remains to present a proof of Proposition \ref{4p1}, which will take up the rest of the paper. In Section \ref{4s2}, by introducing a new labeling system for the branching particle system, we find an appropriate upper bound for $\E(|\widetilde{Z}_{k}|)$ and finish the proof of Proposition \ref{4p1} by assuming a technical lemma (Lemma \ref{4l2.1}) that gives the convergence of the collision term. The proof of Lemma \ref{4l2.1} is then given in Section \ref{4s3} by three intermediate lemmas while Sections \ref{4s4}, \ref{4s5}, \ref{4s6} are devoted to the proofs of those three lemmas.

\section{Moment bounds for the self-avoiding branching random walk} \label{4s2}

We will prove Proposition \ref{4p1} that gives the first moment bound for the self-avoiding branching random walk $\widetilde{Z}$. To do so, we introduce a new labeling system for our particle system. Let \[\cI=\bigcup_{n=0}^\infty  \{0\}\times \{1,\cdots, V(R)\}^n,\] where we use $0$ to denote the ancestor located at the origin. We collect various notations for the labeling system that will be used frequently below:
\begin{itemize}
\item If $\beta=(\beta_0, \beta_1, \cdots, \beta_n)\in \cI$, we set $|\beta|=n$ to be the generation of $\beta$. 
\item Write $\beta|k=(\beta_0, \cdots, \beta_k)$ for each $0\leq k\leq |\beta|$. 
\item For each $|\beta|=n$ with some $n\geq 1$, let $\pi \beta=(\beta_0, \beta_1, \cdots, \beta_{n-1})$ be the parent of $\beta$ and set $\beta \vee i=(\beta_0, \beta_1, \cdots, \beta_n, i)$ to be the $i$-th offspring of $\beta$ for $1\leq i\leq V(R)$.
\item Write $ \beta\geq \gamma$ if $\beta$ is a descendant of $\gamma$ and $\beta>\gamma$ if it is strict.
\item For each $\beta, \gamma \in \cI$, let $k_{max}=\max\{0\leq k\leq |\beta|\wedge |\gamma|: \beta|k=\gamma|k\}$ and define $\gamma\wedge \beta=\beta|k_{max}=\gamma|k_{max}$ to be the most recent common ancestor of $\beta$ and $\gamma$. 
\end{itemize}

Let $\{{W}^{\beta \vee i}, 1\leq i\leq V(R)\}_{\beta \in \cI}$ be a collection of i.i.d. random vectors, each uniformly distributed on  $\cN(0)^{(V(R))}=\{(x_1,\cdots, x_{V(R)}): \{x_i\} \text{ all distinct}\}$. Let $\{{B}^\beta: \beta \in \cI, |\beta|>0\}$ be i.i.d. Bernoulli random variables with parameter $p(R)$ indicating whether the birth of the offspring particle $\beta$ from its parent $\pi \beta$ to  is valid. Let $\{B^{\beta}\}$ and $\{W^\beta\}$ be mutually independent. Define the above independent collections of random variables on some complete probability space $(\Omega, \cF, \P)$.

 Write $\beta \approx n$ if $|\beta|=n$ and ${B}^{\beta|i}=1$ for all $1\leq i\leq n$. The historical path of a particle $\beta$ alive at time $|\beta|$ would be ${Y}_k^\beta= \sum_{i=1}^{|\beta|} 1(i\leq k) W^{\beta|i}$ for $k\geq 0$, and we denote its current location by
\begin{align}\label{5e7.31}
{Y}^\beta=
\begin{cases}
 \sum_{i=1}^{|\beta|}  W^{\beta|i}, &\text{ if } \beta\approx |\beta|,\\
\Delta, &\text{ otherwise. }
\end{cases}
\end{align}
 Set
\begin{align}\label{5e1.24}
\zeta_\beta^0:=\inf\{1\leq m\leq |\beta|: B^{\beta|m}=0\}\wedge (|\beta|+1),
\end{align}
where by convention $\inf \emptyset=\infty$. One can easily check $\zeta_\beta^0>|\beta|$ iff $Y^\beta\neq \Delta$. For each particle $\beta \in \cI$, we denote by $\cH_\beta$ the $\sigma$-field of all the events in the family line of $\beta$ before time $|\beta|$, which is given by
\begin{align}\label{5e11.24}
\cH_\beta=\sigma\{B^{\beta\vert m}, W^{\beta\vert m}: m\leq \vert \beta\vert \}.
\end{align}
Then $Y^\beta\in \cH_\beta$ and $\zeta_\beta^0\in \cH_\beta$. For each $n\geq 1$, define
\begin{align*}
\cG_n=\bigvee_{|\beta|\leq n}\cH_\beta
\end{align*}
to be the $\sigma$-field of all the events before time $n$.

For any function $\phi$, we define
\begin{align}\label{5e1.4}
Z_n(\phi)=\sum_{|\beta|=n} \phi(Y^\beta) 1_{\{\zeta_\beta^0>|\beta|\}}
\end{align}
so that $(Z_n)$ gives the same distribution of the branching random walk $Z$ as in \eqref{4eb2.21} where in generation $n$, each particle gives birth to one offspring to its $V(R)$ neighboring positions independently with probability $p(R)$.

Next, we will use the new labeling system to rewrite the self-avoiding BRW $\widetilde{Z}$ from \eqref{4eb2.24}. To do so, we set
\begin{align}\label{5e11.4}
\zeta_\beta^1=\zeta_\beta^1(\widetilde{Z})=\inf\{1\leq m\leq |\beta|: Y^{\beta|m} \in S(\sum_{k=0}^{m-1} \widetilde{Z}_{k}) \text{ or } B^{\beta|m}=0 \}\wedge (|\beta|+1).
\end{align}
 In this way, the event $\zeta_\beta^1>|\beta|$ is equivalent to $Y^\beta\neq \Delta$ and $Y^{\beta|m} \notin S(\sum_{k=0}^{m-1} \widetilde{Z}_{k})$ for all $1\leq m\leq |\beta|$, that is, the particle $\beta$ is alive at time $|\beta|$ and it never collides with $\widetilde{Z}$ up to time $|\beta|$. 
So the self-avoiding BRW $\widetilde{Z}$ from \eqref{4eb2.24} can be rewritten by
\[
\widetilde{Z}_n(\phi)=\sum_{|\beta|=n} \phi(Y^\beta) 1_{\{\zeta_\beta^1>|\beta|\}}.
\]
Since we start with only one ancestor at the origin, one can check that the above definition is well-given and the existence and uniqueness for the law of $\widetilde{Z}_n$ also follows.

 For any $n\geq 0$, one can check that
\begin{align*}
\widetilde{Z}_{n+1}(1)=&\sum_{|\beta|=n+1} 1_{\{\zeta_\beta^1>|\beta|\}}=\sum_{|\beta|=n}\sum_{i=1}^{V(R)}  1_{\{\zeta_{\beta\vee i}^1>|{\beta\vee i}|\}}\\
=&\sum_{|\beta|=n} 1_{\{\zeta_\beta^1>|\beta|\}} \sum_{i=1}^{V(R)}  B^{\beta \vee i} \cdot 1{\{Y^\beta+W^{\beta \vee i} \notin S(\sum_{k=0}^{n} \widetilde{Z}_{k})\}},
\end{align*}
where the last equality uses the definition of $\zeta_{\beta\vee i}^1$ and $Y^{\beta\vee i}=Y^\beta+W^{\beta \vee i}$. To simplify notation, we set
\begin{align*}
R_n=R_n(\widetilde{Z}):=S\Big(\sum_{k=0}^{n} \widetilde{Z}_{k}\Big).
\end{align*}
Recall $\widetilde{Z}_n(1)=\sum_{|\beta|=n} 1_{\{\zeta_\beta^1>|\beta|\}}$. It follows that 
\begin{align*} 
&\widetilde{Z}_{n+1}(1)-\widetilde{Z}_{n}(1)=\Big[\widetilde{Z}_{n+1}(1)-(1+\frac{\theta}{N_R}) \widetilde{Z}_{n}(1)\Big]+\frac{\theta}{N_R} \widetilde{Z}_{n}(1)\nn\\
=&\sum_{|\beta|=n} 1_{\{\zeta_\beta^1>|\beta|\}} \sum_{i=1}^{V(R)} \Big[ B^{\beta \vee i}1{\{Y^\beta+W^{\beta \vee i} \notin R_n\}}- \frac{1+\frac{\theta}{N_R}}{V(R)}\Big]+\frac{\theta}{N_R} \widetilde{Z}_{n}(1)\nn\\
=&\sum_{|\beta|=n} 1_{\{\zeta_\beta^1>|\beta|\}} \sum_{i=1}^{V(R)} \Big[ B^{\beta \vee i}1{\{Y^\beta+W^{\beta \vee i} \notin R_n\}}- B^{\beta \vee i}\Big]\nn\\
&\quad +\sum_{|\beta|=n} 1_{\{\zeta_\beta^1>|\beta|\}} \sum_{i=1}^{V(R)} [B^{\beta \vee i} -p(R)] +\frac{\theta}{N_R} \widetilde{Z}_{n}(1).
\end{align*}
Now take expectation and use that $B^{\beta \vee i}$ with $|\beta|=n$ is Bernoulli with parameter $p(R)$ independent of $\cG_n \vee \sigma(W^{\beta\vee i})$ to arrive at
\begin{align}\label{5e1.05}
\E(\widetilde{Z}_{n+1}(1)-\widetilde{Z}_{n}(1))= -p(R)\cdot \E\Big(\sum_{|\beta|=n} &1_{\{\zeta_\beta^1>|\beta|\}}  \sum_{i=1}^{V(R)}  1{\{Y^\beta+W^{\beta \vee i} \in R_n\}}  \Big)\nn\\
&+\frac{\theta}{N_R} \E(\widetilde{Z}_{n}(1)).
\end{align}
 For each $|\beta|=n$ and $1\leq i\leq V(R)$, we may condition on $\cG_n$ to see
\begin{align}\label{5e1.20}
  \E(1_{\{Y^\beta+W^{\beta \vee i} \in R_n\}}|\cG_n)=& \sum_{k=1}^{V(R)}1_{\{Y^\beta+e_k\in R_{n}\}} \E\Big(1_{W^{\beta \vee i}=e_k} \Big|\cG_n\Big)\nn\\
=&\frac{1}{V(R)}\sum_{k=1}^{V(R)}1_{\{Y^\beta+e_k\in R_{n}\}}:=\frac{1}{V(R)} \nu(\beta),
\end{align}
where $\nu(\beta)$ counts the number of neighbours of $Y^\beta$ that lie in $R_n$. So \eqref{5e1.05} becomes
\begin{align}\label{5e1.5}
\E(\widetilde{Z}_{n+1}(1)-\widetilde{Z}_{n}(1))=&\frac{\theta}{N_R} \E(\widetilde{Z}_{n}(1))- p(R)\E\Big(\sum_{|\beta|=n} 1_{\{\zeta_\beta^1>|\beta|\}} \sum_{i=1}^{V(R)} \frac{1}{V(R)}   \nu(\beta)\Big)\nn\\
=&\frac{\theta}{N_R} \E(\widetilde{Z}_{n}(1))- p(R)\E\Big(\sum_{|\beta|=n} 1_{\{\zeta_\beta^1>|\beta|\}}  \nu(\beta)\Big).
\end{align}
Summing the above for $0\leq n\leq [N_R]$, we get
\begin{align}\label{5e1.6}
\E(\widetilde{Z}_{[N_R]+1}(1))-1&=\frac{\theta}{N_R} \sum_{n=0}^{[N_R]} \E(\widetilde{Z}_{n}(1))-p(R) \E\Big(\sum_{n=0}^{[N_R]}\sum_{|\beta|=n} 1_{\{\zeta_\beta^1>|\beta|\}}  \nu(\beta) \Big).
\end{align}

Next, recall that 
\begin{align*}
R_{n}=S\Big(\sum_{k=0}^{n} \widetilde{Z}_{k}\Big)=\{Y^\gamma: |\gamma|\leq n, \zeta_{\gamma}^{1}>|\gamma|\}.
\end{align*} 
One may use the above to rewrite $\nu(\beta)$ (for $|\beta|=n$)  as 
\begin{align}\label{5e1.36}
\nu(\beta)=\sum_{k=1}^{V(R)}1_{\{Y^\beta+e_k\in R_{n}\}}=\Big|\{Y^\gamma: |\gamma|\leq |\beta|, \zeta_{\gamma}^{1}>|\gamma|, Y^\beta-Y^\gamma\in \cN(0)\}\Big|.
\end{align}

For each $R>0$, define the cutoff time $\{\tau_R \}$ to be\footnote{The definitions of \eqref{5e2.10} and \eqref{5e3.1} are inspired by Section 5 of \cite{DP99} where the two authors claim that ``collisions between distant relatives can be ignored''. The difference between $\nu(\beta)$ and $\nu_{\tau}(\beta)$ as above can also be ignored, but since we only need an upper bound for \eqref{5e1.6}, we simply replace $\nu(\beta)$ by $\nu_{\tau}(\beta)$.}
\begin{align}\label{5e2.10}
\begin{cases} 
\tau_R \in \N, \tau_R \to \infty, \tau_R/N_R \to 0, &\text{ in } d\geq 5;\\
\tau_R=[N_R/\log N_R], &\text{ in } d=4,
\end{cases}  
\end{align}
and set
 \begin{align}\label{5e3.1}
\nu_{\tau}(\beta)=\Big|\{Y^\gamma: |\gamma|\leq |\beta|, \zeta_{\gamma}^1>|\gamma|, Y^\beta-Y^\gamma\in \cN(0), |\gamma \wedge \beta|> |\beta|-\tau_R \}\Big|.
\end{align}

It is clear that  $\nu(\beta)\geq \nu_{\tau}(\beta)$ and $p(R)\geq 1/V(R)$. Delete the sum for $n<\tau_R$ in the second term of \eqref{5e1.6} to see
\begin{align}\label{5e1.7}
\E(\widetilde{Z}_{[N_R]+1}(1))-1\leq & \frac{\theta}{N_R} \sum_{n=0}^{[N_R]} \E(\widetilde{Z}_{n}(1))-\E\Big( \frac{1}{V(R)}\sum_{n=\tau_R}^{[N_R]}\sum_{|\beta|=n} 1_{\{\zeta_\beta^1>|\beta|\}}  \nu_\tau(\beta) \Big).
\end{align}
Define
\begin{align}\label{5e1.8}
K^1_R:= \frac{1}{V(R)}\sum_{n=\tau_R}^{[N_R]}\sum_{|\beta|=n} 1_{\{\zeta_\beta^1>|\beta|\}}  \nu_\tau(\beta) .
\end{align}
We will show that 
\begin{lemma}\label{4l2.1}
\begin{align}\label{5e1.9}
\lim_{R\to \infty} \Big|\E\Big(K^1_R-\frac{b_d}{N_R} \sum_{n=0}^{[N_R]-\tau_R} \widetilde{Z}_{n}(1)\Big)\Big|=0.
\end{align}
\end{lemma}

Assuming the above lemma, we may finish the proof of Proposition \ref{4p1}.
\begin{proof}[Proof of Proposition \ref{4p1} ]
Set $\eps_0=\frac{1}{4} \wedge \frac{b_d-\theta}{2}>0$. By \eqref{5e2.10} we have 
\begin{align}\label{5e2.2}
\frac{\tau_R}{N_R}\leq \frac{1}{2} \wedge \frac{\eps_0}{4b_d e^{b_d}}, \text{  if $R$ is large.}
\end{align}
Next, Lemma \ref{4l2.1} implies that for $R$ large, 
\begin{align}\label{5e2.0}
\E(K^1_R)\geq \frac{b_d}{N_R} \sum_{n=0}^{[N_R]-\tau_R} \E(\widetilde{Z}_{n}(1))-\frac{\eps_0}{4}.
\end{align}
It follows from \eqref{5e1.7}, \eqref{5e1.8} and the above that
\begin{align}\label{5e2.1}
\E(\widetilde{Z}_{[N_R]+1}(1))-1\leq & \frac{\theta-b_d}{N_R}  \sum_{n=0}^{[N_R]-\tau_R}\E(\widetilde{Z}_{n}(1))+\frac{\eps_0}{4}+\frac{\theta}{N_R} \sum_{n=[N_R]-\tau_R+1}^{[N_R]} \E(\widetilde{Z}_{n}(1))\nn\\
\leq& \frac{-2\eps_0}{N_R}  \sum_{n=0}^{[N_R]-\tau_R}\E(\widetilde{Z}_{n}(1))+\frac{\eps_0}{4}+ \frac{b_d }{N_R}  \tau_R e^{b_d}\nn\\
\leq&\frac{-2\eps_0}{N_R}  \sum_{n=0}^{[N_R]-\tau_R}\E(\widetilde{Z}_{n}(1))+\frac{\eps_0}{2},
\end{align}
where the second inequality uses $b_d-\theta \geq 2\eps_0>0$ and  $\E(\widetilde{Z}_{n}(1))\leq \E({Z}_{n}(1))=(1+\frac{\theta}{N_R})^n  \leq e^\theta\leq e^{b_d}$. The last inequality is by \eqref{5e2.2}.

If there exists some $1\leq n\leq [N_R]$ such that $\E(\widetilde{Z}_{n}(1))\leq 1-\eps_0$, then we are done. If $\E(\widetilde{Z}_{n}(1))\geq 1-\eps_0$ for all $0\leq n\leq [N_R]$, then \eqref{5e2.1} becomes
\begin{align} 
\E(\widetilde{Z}_{[N_R]+1}(1))-1\leq&  {-2\eps_0}  (1-\eps_0)+\frac{\eps_0}{2} \leq - {\eps_0},
\end{align}
where the last inequality uses $\eps_0\leq 1/4$. So the above implies $\E(\widetilde{Z}_{[N_R]+1}(1))\leq 1- {\eps_0}$ as required. In either case, the conclusion of Proposition \ref{4p1} holds. The proof is now complete.
 \end{proof}
It remains to prove Lemma \ref{4l2.1}.

\section{Convergence of the collision term} \label{4s3}

In this section, we will give the proof of Lemma \ref{4l2.1} in three steps. The number $\nu_{\tau}(\beta)$ defined as in \eqref{5e3.1} counts the number of sites that have been occupied in the neighborhood of $Y^\beta$ by its close relatives, which might have been visited more than once. However, one would expect such events to be rare. So we define
\begin{align}\label{e31.60}
\text{nbr}_{\beta, \gamma}(\tau)= 1(|\gamma|\leq |\beta|, Y^\beta-Y^\gamma\in \cN(0), |\gamma \wedge \beta|> |\beta|-\tau_R),
\end{align}
and set
\begin{align}\label{4e3.60}
K^2_R:=&   \frac{1}{V(R)}   \sum_{n=\tau_R}^{[N_R]} \sum_{|\beta|=n} 1_{\{\zeta_\beta^1>|\beta|\}}\sum_{\gamma}  1_{\zeta_{\gamma}^1>|\gamma|} \text{nbr}_{\beta, \gamma}(\tau).
\end{align}
 The lemma below shows that the difference between $K_R^1$ and $K_R^2$ can be ignored. The proof is deferred to Section \ref{4s5}.

\begin{lemma}\label{4l2.2}
$\lim_{R\to \infty} \E(|K^1_R-K_R^2|)=0$.
\end{lemma}

The second step is to replace $\zeta_\beta^1>|\beta|$ and $\zeta_{\gamma}^1>|\gamma|$ in $K_R^2$ by
\[
\{\zeta_\beta^1> |\beta|-\tau_R, Y^\beta\neq \Delta, \zeta_{\gamma}^1> |\beta|-\tau_R, Y^\gamma\neq \Delta\},
\]
and define
\begin{align}\label{4e5.60}
K_R^{3}:=&   \frac{1}{V(R)}  \sum_{n=\tau_R}^{[N_R]} \sum_{|\beta|=n} 1_{\zeta_\beta^1> |\beta|-\tau_R} \sum_{\gamma}  1_{\zeta_{\gamma}^1> |\beta|-\tau_R}  \text{nbr}_{\beta, \gamma}(\tau).
\end{align}
In the above, $Y^\beta\neq \Delta$ and $Y^\gamma\neq \Delta$ are implicitly given by $Y^\beta-Y^\gamma\in \cN(0)$ in $\text{nbr}_{\beta, \gamma}(\tau)$. The following result will be proved in Section \ref{4s6}.

\begin{lemma}\label{4l2.3}
 $\lim_{R\to \infty} \E(|K^2_R-K_R^3|)=0$.
\end{lemma}

Turning to the third step, we notice that  
\begin{align*}
    &\{\zeta_\beta^1> |\beta|-\tau_R, \zeta_{\gamma}^1> |\beta|-\tau_R,|\gamma \wedge \beta|> |\beta|-\tau_R \}\\
    &=\{\zeta_{\beta \wedge \gamma}^1> |\beta|-\tau_R,|\gamma \wedge \beta|> |\beta|-\tau_R \} =\{\zeta_{\beta }^1> |\beta|-\tau_R,|\gamma \wedge \beta|> |\beta|-\tau_R \}.
\end{align*}
Hence we may rewrite $K_R^{3} $ as (recall $\text{nbr}_{\beta, \gamma}(\tau)$ from \eqref{e31.60})
\begin{align*}
K_R^{3}=& \frac{1}{V(R)} \sum_{n=\tau_R}^{[N_R]} \sum_{|\beta|=n} 1_{\{\zeta_\beta^1> |\beta|-\tau_R\}}   \sum_{\gamma} 1_{|\gamma|\leq |\beta|}  1_{Y^\beta-Y^\gamma\in \cN(0)} 1_{|\gamma \wedge \beta|> |\beta|-\tau_R}.
\end{align*}
For each $\beta$ with $|\beta|=n$, we define
\begin{align*}
 F(\beta,n)=\frac{N_R}{V(R)}\sum_{\gamma} 1_{|\gamma|\leq |\beta|}  1_{Y^\beta-Y^\gamma\in \cN(0)} 1_{|\gamma \wedge \beta|> |\beta|-\tau_R}
\end{align*}
so that 
\begin{align}\label{4e11.42}
K_R^{3}=&\frac{1}{N_R}   \sum_{n=\tau_R}^{[N_R]} \sum_{|\beta|=n} 1_{\{\zeta_\beta^1> n-\tau_R\}}     F(\beta,n).
\end{align}

Let $\cA(k)=\{\alpha: |\alpha|=k, \zeta_\alpha^1>k\}$ be the particles alive at time $k$ in $\widetilde{Z}$. Define
\[
\{\alpha\}_n=\{\beta: \beta\geq \alpha, |\beta|=n, \zeta_\beta^0>|\beta|\}
\]
to be the set of descendants $\beta$ of $\alpha$ that are alive at time $n$. Since all the particles $\beta$ contributing in $K_R^{3}$ are descendants of some $\alpha$ in $\cA(n-\tau_R)$, we get \eqref{4e11.42} becomes
\begin{align}\label{4e11.43}
K_R^{3}&= \frac{1}{N_R}   \sum_{n=\tau_R}^{[N_R]} \sum_{\alpha\in \cA(n-\tau_R)}   \sum_{\beta\in \{\alpha\}_n} F(\beta,n).
\end{align}
Set
\[
\cU_{\alpha}(n)=\sum_{\beta\in \{\alpha\}_n} F(\beta,n), \quad \text{ and } \quad b_d^{\tau_R}= \E(\cU_{0}(\tau_R)),
\]
where $0\in \cI$ is the root index located at the origin.
Then 
\begin{align*}
K_R^{3} = \frac{1}{N_R}   \sum_{n=\tau_R}^{[N_R]} \sum_{\alpha\in \cA(n-\tau_R)}  \cU_{\alpha}(n).
\end{align*}
For each $\tau_R\leq n\leq [N_R]$ and $\alpha\in \cA(n-\tau_R)$, if we condition on $\cG_{n-\tau_R}$, then by the Markov property and translation invariance, we get 
\begin{align*}
&\E(\cU_\alpha(n)|\cG_{n-\tau_R})=\E(\cU_{0}(\tau_R))=b_d^{\tau_R}.
\end{align*}
Take expectations in \eqref{4e11.43} and use the above to arrive at
\begin{align}\label{4e3.10}
\E(K_R^{3}) =&\frac{1}{N_R} \E\Big(\sum_{n=\tau_R}^{[N_R]} \sum_{\alpha\in \cA(n-\tau_R)}   b_d^{\tau_R}\Big)=\frac{b_d^{\tau_R}}{N_R} \E\Big(\sum_{n=\tau_R}^{[N_R]} \sum_{|\alpha|=n-\tau_R}  1(\zeta_\alpha^1>n-\tau_R)\Big)\nn\\
=&\frac{b_d^{\tau_R}}{N_R} \E\Big(\sum_{n=\tau_R}^{[N_R]} \widetilde{Z}_{n-\tau_R}(1)\Big)=\frac{b_d^{\tau_R}}{N_R} \E\Big(\sum_{n=0}^{[N_R]-\tau_R} \widetilde{Z}_{n}(1)\Big).
\end{align}
The final piece is the following result, whose proof will be given in Section \ref{4s4}.
\begin{lemma}\label{4l2.4}
 $\lim_{R\to \infty} b_d^{\tau_R}=b_d$.
\end{lemma}

Now we are ready to finish the proof of Lemma \ref{4l2.1}.
\begin{proof} [Proof of Lemma \ref{4l2.1}]
By \eqref{4e3.10}, we get
\begin{align}
 &\Big|\E\Big(K_R^{3}-\frac{b_d}{N_R} \sum_{n=0}^{[N_R]-\tau_R} \widetilde{Z}_{n}(1)\Big)\Big|=\Big|\E\Big(\frac{b_d^{\tau_R}-b_d}{N_R} \sum_{n=0}^{[N_R]-\tau_R} \widetilde{Z}_{n}(1)\Big)\Big|\nn\\
  &\leq |b_d^{\tau_R}-b_d| \frac{1}{N_R} \sum_{n=0}^{[N_R]-\tau_R} \E(\widetilde{Z}_{n}(1)) \leq |b_d^{\tau_R}-b_d|  e^{\theta} \to 0.
\end{align}
Together with Lemmas \ref{4l2.2} and \ref{4l2.3}, the conclusion follows immediately.  
\end{proof}
 
 It remain to prove Lemmas \ref{4l2.2}-\ref{4l2.4}.

\section{Convergence of the constant} \label{4s4}

We first prove Lemma \ref{4l2.4} in this section. Then we will get an explicit expression for $b_d$ as in Theorem \ref{4t0}. Recall that $b_d^{\tau_R}=\E(\cU_{0}(\tau_R))$ where
\begin{align*}
 \cU_{0}(\tau_R)=\sum_{\beta\in \{0\}_{\tau_R}}   F(\beta,\tau_R)=\frac{N_R}{V(R)}& \sum_{\beta\geq 0} 1_{|\beta|=\tau_R, \zeta_\beta^0>|\beta|}\sum_{\gamma} 1_{|\gamma\wedge \beta|>  0} 1_{|\gamma|\leq |\beta|} 1_{Y^\beta-Y^\gamma\in \cN(0)}\\
   =\frac{N_R}{V(R)}& \sum_{\beta> 0} \sum_{\gamma\geq 0}1_{|\beta|=\tau_R}   1_{|\gamma|\leq |\beta|} 1_{Y^\beta-Y^\gamma\in \cN(0)}.
\end{align*}
In the last equality we have used $\{\zeta_\beta^0>|\beta|\}=\{Y^\beta \neq \Delta\}$ and $\{Y^\beta, Y^\gamma\neq \Delta\}$ is implicit in $\{Y^\beta-Y^\gamma \in \cN(0)\}$.
Notice that $Y^\beta-Y^\gamma \in \cN(0)$ implies $Y^\beta \neq Y^\gamma$, so we cannot have $\gamma=\beta$. Since $|\gamma|\leq |\beta|=\tau_R$, we must have $\gamma$ branches off $\beta$ at time $\tau_R-k$ for some $1\leq k\leq \tau_R$, meaning that if we let $\alpha=\gamma\wedge \beta$, then $|\alpha|=\tau_R-k$ for some $1\leq k\leq \tau_R$. Set $|\gamma|=|\alpha|+m$ for some $0\leq m\leq k$. In this way, we get
\begin{align}\label{4e11.01}
 b_d^{\tau_R}=& \E(\cU_{0}(\tau_R)) =\frac{N_R}{V(R)} \E\Big(\sum_{k=1}^{\tau_R} \sum_{\substack{\alpha: \alpha\geq 0,\\ |\alpha|=\tau_R-k}} 1_{\zeta_\alpha^0>|\alpha|} \sum_{\substack{\beta: \beta\geq \alpha,\\ |\beta|=\tau_R} } \sum_{m=0}^{k}  \sum_{\substack{\gamma: \gamma\geq \alpha,\\ |\gamma|=\tau_R-k+m}} 1_{Y^\beta-Y^\gamma\in \cN(0)}\Big)\nn\\
 =&  \frac{N_R}{V(R)}\E\Bigg[\sum_{k=1}^{\tau_R} \sum_{\substack{\alpha: \alpha\geq 0,\\ |\alpha|=\tau_R-k}} 1_{\zeta_\alpha^0>|\alpha|} \E\Big(\sum_{\substack{\beta: \beta\geq \alpha,\\ |\beta|=\tau_R} } \sum_{m=0}^{k}  \sum_{\substack{\gamma: \gamma\geq \alpha,\\ |\gamma|=\tau_R-k+m}} 1_{Y^\beta-Y^\gamma\in \cN(0)}\Big|\cH_\alpha\Big) \Bigg],
\end{align}
where the last equality use $\zeta_\alpha^0 \in \cH_\alpha$ (recall $\cH_\alpha$ from \eqref{5e11.24}).

Next, for each $\alpha, \beta, \gamma$ as in the above summation, we get from \eqref{5e7.31} that
\[
Y^\beta-Y^\alpha=W^{\beta|(\tau_R-k+1)}+\cdots+W^{\beta|(\tau_R-1)}+W^{\beta|\tau_R}:=U_{k}^R.
\]
and
\[
Y^\gamma-Y^\alpha=W^{\gamma|(\tau_R-k+1)}+\cdots+W^{\gamma|(\tau_R-k+m-1)}+W^{\gamma|(\tau_R-k+m)}:=V_{m}^R.
\]
Set
\begin{align}\label{4e11.12}
W_{k+m}^R:=U_{k}^R+V_{m}^R.
\end{align}
 Recall that $W^{\beta|i}$ and $W^{\gamma|j}$ for each $i,j$ are uniform on $\cN(0)$. Although there is a slight dependence between $W^{\beta|(\tau_R-k+1)}$ and $W^{\gamma|(\tau_R-k+1)}$ (they are uniformly distributed over $\{(x,y)\in\cN(0)^2: x\neq y\}$), it is clear that the joint distribution of $(W^{\beta|(\tau_R-k+1)},W^{\gamma|(\tau_R-k+1)})$ converges to $(U,V)$ where $U,V$ are independent and uniformly distributed on $[-1,1]^d$. So $W_{k+m}^R=U_{k}^R+V_{m}^R$ will converge in distribution to $U_{k+m}$ where $U_n=Y_1+\cdots+Y_n$ and $Y_i$ are i.i.d. uniform on $[-1,1]^d$. \\

Since both $Y^\beta-Y^\alpha$ and $Y^\gamma-Y^\alpha$ are independent of $\cH_\alpha$,  we get on the event $\{\zeta_\alpha^0>|\alpha|\}$ for $|\alpha|=\tau_R-k$,
\begin{align}\label{4e11.02}
 & \E\Big(\sum_{\substack{\beta: \beta\geq \alpha,\\ |\beta|=\tau_R} } \sum_{m=0}^{k}  \sum_{\substack{\gamma: \gamma\geq \alpha,\\ |\gamma|=\tau_R-k+m}}  1_{Y^\beta-Y^\alpha-(Y^\gamma-Y^\alpha)\in \cN(0)}\Big|\cH_\alpha\Big)\\
  & = \sum_{\substack{\beta: \beta\geq \alpha,\\ |\beta|=\tau_R} } \sum_{m=0}^{k}  \sum_{\substack{\gamma: \gamma\geq \alpha,\\ |\gamma|=\tau_R-k+m}}   \P(U_{k}^R+V_{m}^R\in \cN(0)) p(R)^{k} p(R)^m\nn\\
&=  \sum_{m=0}^{k}  V(R)^{k} (V(R)-1)^{m\wedge 1} V(R)^{(m-1)^+}  \P(W_{k+m}^R\in \cN(0))  p(R)^{k+m}.\nn
\end{align}
In the second line, the term $p(R)^{k} p(R)^m$ gives the probability that $Y^\beta, Y^\gamma \neq \Delta$ on the event $\{Y^\alpha\neq \Delta\}$. The third line follows by counting the number of all possible $\beta,\gamma$ as in the summation, where $(V(R)-1)^{m\wedge 1}$ comes from that $\gamma|(\tau_R-k+1) \neq \beta|(\tau_R-k+1)$. 

The remaining sum of $\alpha$ in \eqref{4e11.01} gives 
\begin{align}\label{4e11.03}
&\E\Big(\sum_{k=1}^{\tau_R} \sum_{\substack{\alpha: \alpha\geq 0,\\ |\alpha|=\tau_R-k}} 1_{\zeta_\alpha^0>|\alpha|} \Big)=\sum_{k=1}^{\tau_R} [V(R)p(R)]^{\tau_R-k}.
 \end{align}
 So we conclude from \eqref{4e11.01}, \eqref{4e11.02}, \eqref{4e11.03} that
\begin{align*}
b_d^{\tau_R}=\frac{N_R}{V(R)} \sum_{k=1}^{\tau_R}  [p(R)V(R)]^{\tau_R}  & \sum_{m=0}^{k} \P(W_{k+m}^R\in \cN(0))   \nn\\
&V(R)^{(m-1)^+} (V(R)-1)^{m\wedge 1} p(R)^{m}.
   \end{align*}
Clearly we may replace $V(R)^{(m-1)^+} (V(R)-1)^{m\wedge 1} $ by $V(R)^m$ as  
\[
1\leq \frac{V(R)^m}{V(R)^{(m-1)^+} (V(R)-1)^{m\wedge 1}} \leq e^{\tau_R/(V(R)-1)} \text{ for all $0\leq m\leq k\leq \tau_R$},
\]
and $\tau_R/(V(R)-1)\leq \tau_R/N_R\to 0$ by \eqref{5e2.10}. Hence 
\begin{align*}
\lim_{R\to \infty} b_d^{\tau_R}=\lim_{R\to \infty} \frac{N_R}{V(R)} \sum_{k=1}^{\tau_R}   [V(R)p(R)]^{\tau_R} \sum_{m=0}^{k}  \P(W_{k+m}^R\in \cN(0))   [V(R)p(R)]^{m}.
\end{align*}
Similarly, we may replace $[V(R)p(R)]^{m}$ and $[V(R)p(R)]^{\tau_R}$ by $1$ as 
\[
1\leq [V(R)p(R)]^{m} \leq e^{\tau_R/N_R} \text{ for all $0\leq m\leq \tau_R$},
\]
and so
\begin{align}\label{4e11.13}
 \lim_{R\to \infty} b_d^{\tau_R} =\lim_{R\to \infty} \frac{N_R}{V(R)} \sum_{k=1}^{\tau_R }   \sum_{m=0}^{ k}  \P(W_{k+m}^R\in \cN(0)).
\end{align}

To calculate the above limit, we let $Y_1^R,Y_2^R, \cdots$ be i.i.d  uniform on $\cN(0)$. Define $S_n^R=Y_1^R+\cdots+Y_n^R$ for each $n\geq 1$ and set $S_n^R=0$ for $n\leq 0$. The following lemma comes from (4) in Section 2 of \cite{BDS89} and the concentration inequality by Kesten \cite{Kes69}.
\begin{lemma}\label{9l4.2}
 There is some constant $C>0$ independent of $R$ so that
\begin{align*}
\P(x+S_n^R\in [-1,1]^d)\leq C  (1+n)^{-d/2}, \quad \forall n\geq 0, x\in \R^d.
\end{align*}
\end{lemma}
Recall $W_{k+m}^R$ from \eqref{4e11.12}.  Apply Lemma \ref{9l4.2} to get  
\begin{align}\label{4e11.15}
\P(W_{k+m}^R\in \cN(0))\leq \sup_x \P(x+S_{k+m-1}^R\in [-1,1]^d) \leq \frac{C}{(k+m)^{d/2}}.
\end{align}

\no When $d\geq 5$, we get
\[
\frac{N_R}{V(R)}=\frac{R^d}{(2R+1)^d-1} \to 2^{-d},
\]
and \eqref{4e11.15} implies
\[
 \sum_{k=1}^{\tau_R }   \sum_{m=0}^{ k}  \P(W_{k+m}^R\in \cN(0)) \leq  \sum_{k=1}^{\infty}\sum_{m=0}^{k}  \frac{C}{(k+m)^{d/2}}<\infty.
\]
Therefore by applying Dominated Convergence on the right-hand side of \eqref{4e11.13}, we get
\begin{align*}
\lim_{R\to \infty}  b_d^{\tau_R}  =2^{-d} \sum_{k=1}^{\infty}\sum_{m=0}^{k} \P(U_{k+m}\in [-1,1]^d)=b_d.
\end{align*}

When $d=4$, we have $\frac{V(R)}{N_R} \sim 2^4 \log R$, thus giving
\begin{align}  \label{4e11.10}
\lim_{R\to \infty} b_4^{\tau_R}=\lim_{R\to \infty} \frac{1}{2^4 \log R}  \sum_{k=1}^{\tau_R}   \sum_{m=0}^{k}  \P(W_{k+m}^R\in \cN(0)).
\end{align}
By \eqref{4e11.15}, we have
\begin{align*} 
\sum_{k=1}^{\log R}   \sum_{m=0}^{k}  \P(W_{k+m}^R\in \cN(0))&\leq \sum_{k=1}^{\log R}   \sum_{m=0}^{k} \frac{C}{(k+m)^2}\\
&\leq \sum_{k=1}^{\log R}     \frac{C}{k}\leq C \log \log R=o(\log R).
\end{align*}
Hence we may delete the sum for $1\leq k\leq \log R$ in \eqref{4e11.10} to get
\begin{align}  \label{4e11.85}
\lim_{R\to \infty} b_4^{\tau_R}=\lim_{R\to \infty} \frac{1}{2^4 \log R}  \sum_{k=\log R}^{\tau_R}   \sum_{m=0}^{k}  \P(W_{k+m}^R\in \cN(0)).
\end{align}

 For each $m\geq 0$ and $k\geq 1$, define
 \[
 h_R(k,m)=\P(W_{k+m}^R\in \cN(0))\quad  \text{ and } \quad g(k,m)=2^d (2\pi/3)^{-d/2} (k+m)^{-d/2}.
 \]
 The following is an application of the classical Central Limit Theorem.
\begin{lemma}
(i) If $R\to\infty$ and $x_n/n^{1/2}\to x$ as $n\to \infty$, then for any Borel set with $\vert \partial B\vert =0$ and $\vert B\vert <\infty$, we have
\begin{align*}
n^{d/2}\P(x_n+S_{n}^R\in B) \to \vert B\vert  \cdot  (2\pi\sigma^2)^{-d/2} e^{-{\vert x\vert ^2}/{2\sigma^2}},
\end{align*}
where $\sigma^2=1/3$ is the limit of the variance of one component of $Y_1^R$ as $R\to\infty$.\\
\no (ii) For any $\eps>0$ small, if $R$ is large, then 
\begin{align}  \label{9ec1.85}
h_R(k,m)/g(k,m)\in [1-\eps,1+\eps], \quad  \forall k\geq \log R, m\geq 0.
\end{align}
\end{lemma}
\begin{proof}
By Lemma 4.6 of \cite{BDS89}, we have (i) holds.  For the proof of (ii), we note (i) ensures that 
\[
\lim_{k+m\to \infty} \frac{h_R(k,m)}{g(k,m)}=1.
\]
So if $R$ is large enough, we get $k+m\geq \log R$ is large and hence \eqref{9ec1.85} follows.
\end{proof}
By \eqref{9ec1.85}, we may replace $h_R(k,m)$ in \eqref{4e11.85} by $g(k,m)$ to get
\begin{align}\label{4e7.61}
\lim_{R\to \infty}  b_4^{\tau_R}=\lim_{R\to \infty} \frac{1}{2^4 \log R}  \sum_{k=\log R}^{\tau_R}   \sum_{m=0}^{k}  2^4 (2\pi/3)^{-2} (k+m)^{-2}.
\end{align}
For any $\eps>0$, it is easy to check if $k\geq \log R$ is large, then
\[
 \sum_{m=0}^{k} \frac{1}{(k+m)^{2}}  \in \Big[(1-\eps) \frac{1}{2k}, (1+\eps) \frac{1}{2k}\Big].
 \]
 So replace $ \sum_{m=0}^{k} {(k+m)^{-2}}$ in \eqref{4e7.61} by $\frac{1}{2k}$ to get
\begin{align*}
\lim_{R\to \infty}  b_4^{\tau_R}=\lim_{R\to \infty}  \frac{9}{4\pi^2\log R} \sum_{k=\log R}^{R^4/\log R}    \frac{1}{2k}=\frac{9}{2\pi^2}=b_4,
\end{align*}
as required.

\section{Proof of Lemma \ref{4l2.2}} \label{4s5}

In this section, we will prove Lemma \ref{4l2.2}. Define for each $t\geq 0$
\begin{align}
I(t)=1+\int_0^t (1+s)^{1-d/2} ds.
\end{align}
One can check that there exists some constant $C>0$ so that
\begin{align}\label{4e4.54}
  I(N_R)\leq 
  \begin{dcases}
  C, &\quad \forall R\text{ large in $d\geq 5$},\\
  C \log R,& \quad \forall R\text{ large in $d=4$},
 \end{dcases}
\end{align}
and
\begin{align}\label{4e4.55}
 N_R I(N_R)\leq CV(R).
\end{align}

Recall $K_R^1$ from \eqref{5e1.8} and $K_R^2$ from \eqref{4e3.60} to get
\begin{align*} 
|K^1_R-K_R^2|\leq  \frac{1}{V(R)}\sum_{n=\tau_R}^{[N_R]}\sum_{|\beta|=n} 1_{\{\zeta_\beta^0>|\beta|\}} \Big|\nu_\tau(\beta)-\sum_{\gamma}  1_{\zeta_{\gamma}^1>|\gamma|} \text{nbr}_{\beta, \gamma}(\tau)\Big|,
\end{align*}
where we have replaced $\zeta_\beta^1>|\beta|$ by $\zeta_\beta^0>|\beta|$.
The absolute value on the right-hand side above arises from the multiple occupancies of particles, that is, $\nu_\tau(\beta)$ in $K^1_R$ only counts the number of sites in the neighborhood of $Y^\beta$ that has been occupied whereas the corresponding term in $K^2_R$ counts the total number of particles that have ever visited that neighborhood. We claim that
\begin{align*}
&|\nu_\tau(\beta)-\sum_{\gamma}  1_{\zeta_{\gamma}^1>|\gamma|} \text{nbr}_{\beta, \gamma}(\tau)|\\
&\leq\sum_{\gamma} 1_{|\gamma|\leq |\beta|}  1_{Y^\beta-Y^\gamma\in \cN(0)}1_{|\gamma\wedge \beta|>|\beta|-\tau_R}  \sum_{\alpha\neq \gamma} 1_{Y^\alpha=Y^\gamma} 1_{|\alpha|\leq |\beta|} 1_{|\alpha\wedge \beta|>|\beta|-\tau_R}.
\end{align*}
To see this, if there are $k\geq 2$ particles that have ever visited the site $Y^\gamma$ in the neighborhood of $Y^\beta$, then they will contribute at most $k-1$ to the left-hand side while at least $k(k-1)$ to the right-hand side, thus giving the above. Note $\{Y^\gamma, Y^\alpha \neq \Delta\}$ is implicit in $\{Y^\alpha=Y^\gamma\}$. It follows that
 \begin{align}\label{4e4.58}
\E(|K_R^1-K_R^2|) &\leq   \frac{1}{V(R)}  \E\Big( \sum_{n=\tau_R}^{[N_R]} \sum_{|\beta|=n} 1_{\{\zeta_\beta^0>|\beta|\}} \sum_{\gamma} 1_{|\gamma|\leq |\beta|}  1_{Y^\beta-Y^\gamma\in \cN(0)}\nn \\
&1_{|\gamma\wedge \beta|>|\beta|-\tau_R}  \sum_{\alpha\neq \gamma} 1_{Y^\alpha=Y^\gamma} 1_{|\alpha|\leq |\beta|} 1_{|\alpha\wedge \beta|>|\beta|-\tau_R}\Big):=\frac{1}{V(R)}  \E(J_1).
\end{align}
Since $\alpha$ and $\gamma$ are symmetric, we may assume $\alpha\wedge \beta \leq \gamma\wedge \beta$. There two cases for $\alpha, \beta, \gamma$:
 \begin{align*} 
 &\text{ (i) } \alpha\wedge \beta < \gamma\wedge \beta; \quad \text{(ii) } \alpha\wedge \beta = \gamma\wedge \beta.
\end{align*}
Denote by $J^{(i)}_1$ (resp. $J^{(ii)}_1$) for the contribution to $J_1$ when $\alpha,\beta, \gamma$ satisfy case (i) (resp. case (ii)).  \\

\no {\bf Case (i).} Let $\sigma=\alpha\wedge \beta$ and $\delta=\gamma\wedge \beta$. In case (i) we have $\sigma < \delta$ and $\delta<\beta$ by $Y^\gamma  \neq Y^\beta$.
For each $|\beta|=n$ with $\tau_R\leq n\leq [N_R]$, we let $\sigma=\beta|k$ and $\delta=\beta|j$ for some $n-\tau_R\leq k<j\leq n-1$. Set $|\alpha|=k+m$ for some $0\leq m\leq n-k$ and $|\gamma|=j+l$ for some $0\leq l\leq n-j$. 
See Figure \ref{fig1} below for illustration.\\

\begin{figure}[ht]
  \begin{center}
    \includegraphics[width=0.65 \textwidth]{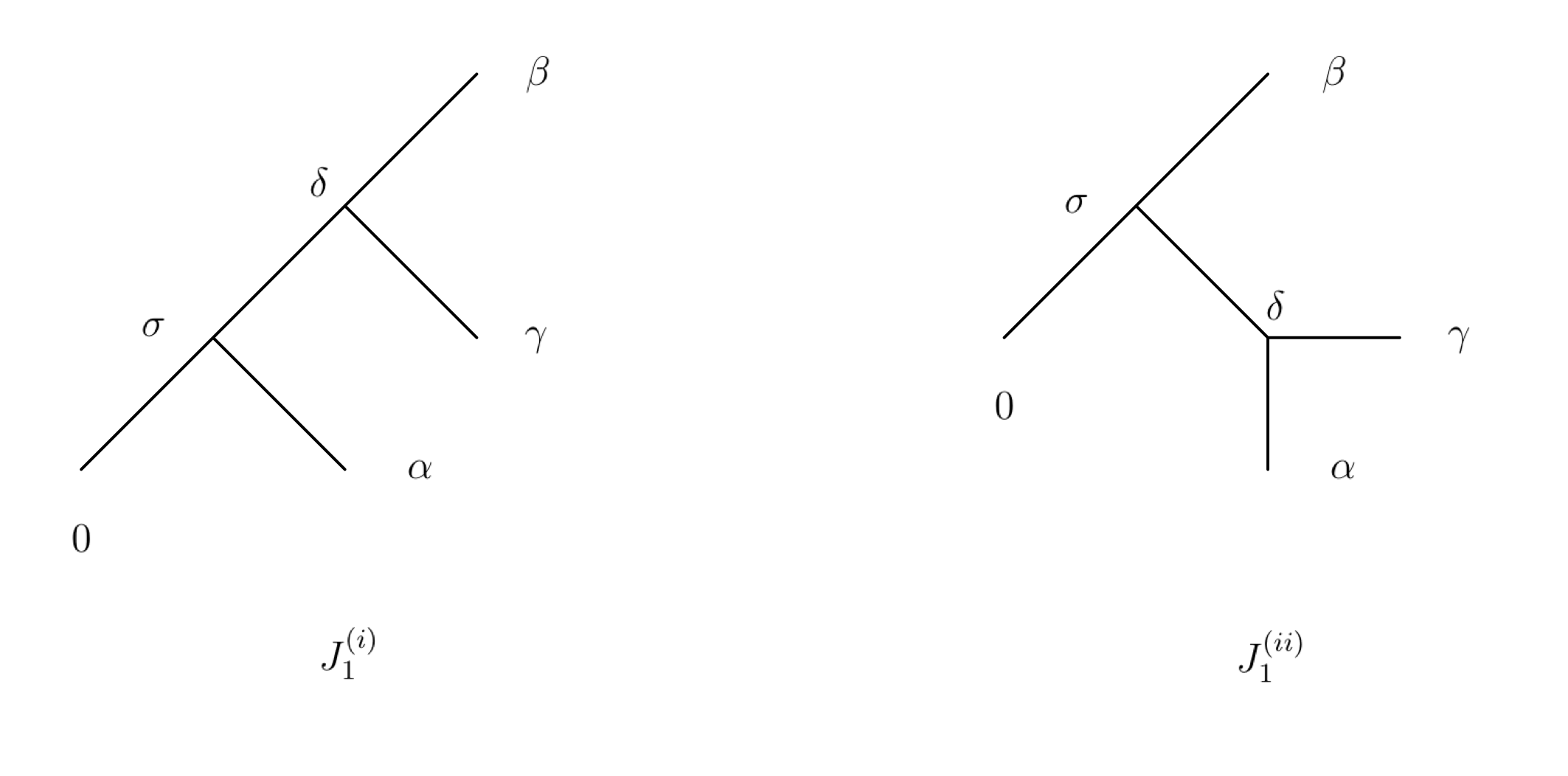}
    \caption[Branching Particle System]{\label{fig1}   Two cases for $J_1$.
      }
  \end{center}
\end{figure}

The sum of $\alpha, \beta, \gamma$ from $J^{(i)}_1$ can be written as
 \begin{align}\label{9e1.21}
\E(J_1^{(i)}) =\E\Big( \sum_{n=\tau_R}^{[N_R]} &  \sum_{k=n-\tau_R}^{n-1} \sum_{j=k+1}^{n-1} \sum_{m=0}^{n-k} \sum_{l=0}^{n-j}\sum_{\substack{\sigma: |\sigma|=k}} \sum_{\substack{\delta: |\delta|=j,\\ \delta \geq \sigma}} \sum_{\substack{\alpha: |\alpha|=k+m,\\ \alpha\geq \sigma}} \sum_{\substack{\gamma: |\gamma|=j+l,\\ \gamma\geq \delta}}\sum_{\substack{\beta: |\beta|=n, \\ \beta\geq \delta}} \nn  \\
& 1_{\{Y^\alpha, Y^\beta, Y^\gamma\neq \Delta\}} 1_{\{Y^\beta-Y^\gamma\in \cN(0)\}} 1_{\{Y^\alpha=Y^\gamma\}}    \Big).
 \end{align} 
 Recall $\cH_\alpha$ from \eqref{5e11.24}. By conditioning on $\cH_\alpha \vee\cH_\gamma$, on the event $\{Y^\alpha, Y^\beta, Y^\gamma\neq \Delta\}$ we get
\begin{align*}
    &\P(Y^\beta-Y^\gamma\in \cN(0)|\cH_\alpha \vee\cH_\gamma)\\
    & =\P\Big(Y^\beta-Y^\delta-W^{\beta|(j+1)}+(Y^\delta+W^{\beta|(j+1)}-Y^\gamma)\in \cN(0)\Big|\cH_\alpha \vee\cH_\gamma\Big).
 \end{align*}
Note that $\delta=\beta|j$. Recall \eqref{5e7.31} to get 
\[
Y^\beta-Y^\delta-W^{\beta|(j+1)}=\sum_{t=j+2}^n  W^{\beta|t},
\]
which is independent of $\cH_\alpha \vee\cH_\gamma$. Hence we get  
 \begin{align}\label{4e1.22}
    &\P(Y^\beta-Y^\gamma\in \cN(0)|\cH_\alpha \vee\cH_\gamma) \leq  \sup_x \P\Big(\sum_{t=j+2}^n  W^{\beta|t}+x\in [-1,1]^d\Big)\leq   \frac{C}{(n-j)^{d/2}}, 
 \end{align}  
where the last inequality is by Lemma \ref{9l4.2}.

Next, recall $|\alpha\wedge \gamma|=|\sigma|=k$ with $|\alpha|=k+m$ and $|\gamma|=j+l$. Use \eqref{5e7.31} again to get
\[
Y^\alpha-Y^\gamma=\sum_{t=k+1}^{k+m} W^{\alpha|t}+\sum_{s=k+1}^{j+l} W^{\gamma|s}.
\]
Notice that $W^{\gamma|(j+l)}$ is independent of everything else. Use the total probability formula and let $W^{\gamma|(j+l)}=e_i$ for $1\leq i\leq V(R)$ to obtain
 \begin{align}\label{4e1.23}
    &\P(Y^\alpha=Y^\gamma)=\frac{1}{V(R)} \sum_{i=1}^{V(R)} \P(Y^\alpha-Y^\gamma-W^{\gamma|(j+l)}=-e_i)\nn\\
    &=\frac{1}{V(R)} \P\Big(\sum_{t=k+1}^{k+m} W^{\alpha|t}+\sum_{s=k+1}^{j+l-1} W^{\gamma|s}\in \cN(0)\Big) \leq \frac{1}{V(R)} \frac{C}{(j+l-k+m)^{d/2}}.
\end{align}

Combine \eqref{4e1.22} and \eqref{4e1.23} to see \eqref{9e1.21} becomes
\begin{align*} 
\E(J_1^{(i)})\leq& \sum_{n=\tau_R}^{[N_R]}   \sum_{k=n-\tau_R}^{n-1} \sum_{j=k+1}^{n-1} \sum_{m=0}^{n-k} \sum_{l=0}^{n-j}\sum_{\substack{\sigma: |\sigma|=k}} \sum_{\substack{\delta: |\delta|=j,\\ \delta \geq \sigma}} \sum_{\substack{\alpha: |\alpha|=k+m,\\ \alpha\geq \sigma}} \sum_{\substack{\gamma: |\gamma|=j+l,\\ \gamma\geq \delta}}\sum_{\substack{\beta: |\beta|=n, \\ \beta\geq \delta}}   \\
&\quad \P(Y^\alpha, Y^\beta, Y^\gamma\neq \Delta)  \frac{C}{(n-j)^{d/2}}\frac{1}{V(R)}  \frac{C}{(l+(j-k)+m)^{d/2}}.
\end{align*}
The probability $\P(Y^\alpha, Y^\beta, Y^\gamma\neq \Delta)$ is bounded by $p(R)^{k} p(R)^{j-k}p(R)^{m} p(R)^{l}  p(R)^{n-j}$ while the sum of $\sigma, \delta, \alpha, \gamma, \beta$ gives $V(R)^{k} V(R)^{j-k}V(R)^{m} V(R)^{l}  p(R)^{n-j}$. So the above is at most
\begin{align*} 
\E(J_1^{(i)})\leq&\sum_{n=\tau_R}^{[N_R]}   \sum_{k=n-\tau_R}^{n-1} \sum_{j=k+1}^{n-1} \sum_{m=0}^{n-k} \sum_{l=0}^{n-j} (V(R)p(R))^{k}  (V(R)p(R))^{j-k}  (V(R)p(R))^{m}    \\
&\quad (V(R)p(R))^{l}   (V(R)p(R))^{n-j}  \frac{C}{(n-j)^{d/2}}\frac{1}{V(R)}  \frac{C}{(l+(j-k)+m)^{d/2}}.
\end{align*}
Use $k+(j-k)+m+l+(n-j)\leq 3n\leq 3[N_R]$ and $V(R)p(R)\leq e^{\theta/N_R}$ to get the above can be bounded by
 \begin{align*}
C e^{3\theta }\frac{1}{V(R)}  \sum_{n=\tau_R}^{[N_R]}   \sum_{k=n-\tau_R}^{n-1} \sum_{j=k+1}^{n-1} \sum_{m=0}^{n-k} \sum_{l=0}^{n-j}  \frac{1}{(n-j)^{d/2}}\frac{1}{(l+(j-k)+m)^{d/2}}.
 \end{align*} 
The sum of $l$ gives at most $C/(1+(j-k)+m)^{d/2-1}$ and then the sum of $m$ gives at most 
\[
\sum_{m=0}^{n-k} \frac{C}{(1+(j-k)+m)^{d/2-1}}\leq \sum_{m=0}^{N_R} \frac{1}{(1+m)^{d/2-1}} \leq I(N_R).
\]
Next, the sum of $j$ gives
\[
\sum_{j=k+1}^{n-1}  \frac{1}{(n-j)^{d/2}}=\sum_{j=1}^{n-k-1}  \frac{1}{j^{d/2}}\leq C.
\]
Combine the above to see
 \begin{align}\label{4e4.57}
\E(J_1^{(i)})\leq&  C e^{3\theta }\frac{1}{V(R)}\sum_{n=\tau_R}^{[N_R]}   \sum_{k=n-\tau_R}^{n-1}  CI(N_R)\leq C\frac{1}{V(R)}I(N_R) \tau_R N_R \leq C\tau_R,
 \end{align}
 where the last inequality uses \eqref{4e4.55}. \\

\no {\bf Case (ii)}: In this case we have $\alpha\wedge \beta=\gamma \wedge \beta$. Let $\sigma=\alpha\wedge \beta$.
For each $|\beta|=n$ with $\tau_R\leq n\leq [N_R]$, we let $\sigma=\beta|k$ for some $n-\tau_R\leq k\leq n-1$ as we assume $|\alpha\wedge \beta|\geq  |\beta|-\tau_R$. Let $\delta=\alpha \wedge \gamma$. Then $\delta\geq \sigma$ and we set $|\delta|=|\sigma|+j=k+j$ for some $0\leq j\leq n-k$. Let $|\alpha|=k+j+m$ and $|\gamma|=k+j+l$ for some $0\leq m,l\leq n-k-j$. See Figure \ref{fig1} for the illustration for $J^{(ii)}_1$. The sum of $\alpha, \beta, \gamma$ in $J^{(ii)}_1$ can be written as
 \begin{align}\label{9e1.24}
\E(J_1^{(ii)})= \E\Big(\sum_{n=\tau_R}^{[N_R]} &  \sum_{k=n-\tau_R}^{n-1} \sum_{j=0}^{n-k} \sum_{m=0}^{n-k-j} \sum_{l=0}^{n-k-j}\sum_{\substack{\sigma: |\sigma|=k}} \sum_{\substack{\beta: |\beta|=n, \\ \beta\geq \sigma}} \sum_{\substack{\delta: |\delta|=k+j,\\ \delta \geq \sigma}} \sum_{\substack{\alpha: |\alpha|=k+j+m,\\ \alpha\geq \delta}} \sum_{\substack{\gamma: |\gamma|=k+j+l,\\ \gamma\geq \delta}}  \nn \\
&1_{\{Y^\alpha, Y^\beta, Y^\gamma\neq \Delta\}} 1_{\{Y^\beta-Y^\gamma\in \cN(0)\}} 1_{\{Y^\alpha=Y^\gamma\}}    \Big).
 \end{align} 
 Similar to the derivation of \eqref{4e1.22}, one may get that on the event $\{Y^\alpha, Y^\beta, Y^\gamma\neq \Delta\}$, 
\begin{align*}
&\P(Y^\beta-Y^\gamma\in \cN(0)|\cH_\alpha \vee \cH_\gamma)\\
&=\P\Big(Y^\beta-Y^\sigma-W^{\beta|(k+1)}+(Y^\sigma+W^{\beta|(k+1)}-Y^\gamma)\in \cN(0)\Big|\cH_\alpha \vee \cH_\gamma\Big) \leq \frac{C}{(n-k)^{d/2}}.
\end{align*}
Also similar to \eqref{4e1.23}, we obtain
\begin{align*}
\P(Y^\alpha=Y^\gamma) \leq \frac{1}{V(R)}  \frac{C}{(1+l+m)^{d/2}}.
\end{align*}
So \eqref{9e1.24} becomes
\begin{align*}
\E(J_1^{(ii)})\leq& \sum_{n=\tau_R}^{[N_R]}   \sum_{k=n-\tau_R}^{n-1} \sum_{j=0}^{n-k} \sum_{m=0}^{n-k-j} \sum_{l=0}^{n-k-j}\sum_{\substack{\sigma: |\sigma|=k}} \sum_{\substack{\beta: |\beta|=n, \\ \beta\geq \sigma}} \sum_{\substack{\delta: |\delta|=k+j,\\ \delta \geq \sigma}} \sum_{\substack{\alpha: |\alpha|=k+j+m,\\ \alpha\geq \delta}} \sum_{\substack{\gamma: |\gamma|=k+j+l,\\ \gamma\geq \delta}}   \\
&\quad \P(Y^\alpha, Y^\beta, Y^\gamma\neq \Delta) \frac{C}{(n-k)^{d/2}} \frac{1}{V(R)} \frac{C}{(1+l+m)^{d/2}}\\
=&\sum_{n=\tau_R}^{[N_R]}   \sum_{k=n-\tau_R}^{n-1} \sum_{j=0}^{n-k} \sum_{m=0}^{n-k-j} \sum_{l=0}^{n-k-j} (V(R)p(R))^{k}  (V(R)p(R))^{n-k}  (V(R)p(R))^{j}    \\
&\quad (V(R)p(R))^{l}   (V(R)p(R))^{m}  \frac{C}{(n-k)^{d/2}} \frac{1}{V(R)} \frac{C}{(1+l+m)^{d/2}}.
 \end{align*} 
Use $k+(n-k)+j+m+l\leq 3n\leq 3[N_R]$ and $V(R)p(R)\leq e^{\theta/N_R}$ to get the above is at most
 \begin{align*}
C e^{3\theta }\frac{1}{V(R)}  \sum_{n=\tau_R}^{[N_R]}   \sum_{k=n-\tau_R}^{n-1} \sum_{j=0}^{n-k} \sum_{m=0}^{n-k-j} \sum_{l=0}^{n-k-j}  \frac{1}{(n-k)^{d/2}}   \frac{1}{(1+l+m)^{d/2}}.
 \end{align*} 
The sum of $l$ is bounded by $C/(1+m)^{d/2-1}$ and then the sum of $m$ gives at most $CI(N_R)$. Next, the sum of $j$ gives $n-k+1\leq 2(n-k)$ and we are left with
\begin{align*}
Ce^{3\theta } \frac{1}{V(R)}  I(N_R) \sum_{n=\tau_R}^{[N_R]}   \sum_{k=n-\tau_R}^{n-1}     \frac{1}{(n-k)^{d/2-1}}.
 \end{align*} 
 The sum of $k$ above is bounded by $I(N_R)$ and the sum of $n$ gives at most $N_R$. Combine the above to see
 \begin{align}\label{4e4.56}
\E(J_1^{(ii)})\leq&  C \frac{1}{V(R)} I(N_R)^2  N_R\leq CI(N_R),
 \end{align} 
  where the last inequality uses \eqref{4e4.55}. 
  
  We are ready to finish the proof of Lemma \ref{4l2.2}.
  
 \begin{proof}[Proof of Lemma \ref{4l2.2}]
Apply \eqref{4e4.58}, \eqref{4e4.57}, \eqref{4e4.56} to get
 \begin{align*}
\E(|K_R^1-K_R^2|) \leq&  \frac{1}{V(R)}  [\E(J_1^{(i)})+\E(J_1^{(ii)})] 
\leq \frac{1}{V(R)} \Big[C\tau_R+CI(N_R)\Big].
\end{align*}
 When $d\geq 5$, we have $I(N_R)\leq C$ and $\tau_R/R^d \to 0$ by \eqref{5e2.10}, so
 \begin{align*}
\E(|K_R^1-K_R^2|) \leq    C\frac{\tau_R}{R^d}+C \frac{1}{V(R)} \to 0.
 \end{align*} 
 When $d=4$, we get $I(N_R)\leq C \log R$ and $\tau_R=N_R/\log N_R \leq R^4/(\log R)^2$, so
  \begin{align*}
\E(|K_R^1-K_R^2|) \leq    C\frac{1}{(\log R)^2}+  C \frac{\log R}{R^4} \to 0.
 \end{align*} 
 The proof is now complete.
  \end{proof}

\section{Proof of Lemma \ref{4l2.3}} \label{4s6}

The last section is devoted to the proof of Lemma \ref{4l2.3}. Recall $K_R^{2}$ from \eqref{4e3.60} and $K_R^3$ from \eqref{4e5.60} to see
\begin{align} \label{9e8.01}
\E( \vert K_R^{2}-K_R^{3}\vert) &\leq    \frac{1}{V(R)} \E\Big(\sum_{n=\tau_R}^{[N_R]}\sum_{|\beta|=n}  \sum_{\gamma: |\gamma|\leq |\beta|}     1_{Y^\beta-Y^\gamma \in \cN(0)}  1_{|\gamma \wedge \beta|> |\beta|-\tau_R}\nn\\
&\times \Big[1_{\zeta_\beta^1> |\beta|-\tau_R} 1_{\zeta_{\gamma}^1> |\beta|-\tau_R}-1_{\zeta_\beta^1> |\beta|} 1_{\zeta_\gamma^{1}> |\gamma|}\Big]\Big).
\end{align}
Since $|\gamma|\leq |\beta|$, we get $1_{\zeta_{\gamma}^1> |\beta|-\tau_R}\leq 1_{\zeta_{\gamma}^1> |\gamma|-\tau_R}$. Bound the above square bracket by 
\begin{align*}
 1_{\{  |\beta|-\tau_R < \zeta_\beta^1\leq |\beta|\}} +1_{\{ |\gamma|-\tau_R< \zeta_\gamma^1 \leq |\gamma|\}}.
\end{align*}
Next, to get symmetry between $\beta, \gamma$, we let $\alpha=\beta\wedge \gamma$. Notice that in the above sum, $\alpha, \beta, \gamma$ satisfy $\tau_R\leq |\beta|\leq [N_R]$, $|\gamma|\leq |\beta|$ and $|\alpha| \geq |\beta|-\tau_R$. One can easily deduce that $0\leq |\alpha|\leq  [N_R]$, $\beta\geq \alpha$, $\gamma \geq \alpha$, $|\beta|\leq |\alpha|+\tau_R$ and $|\gamma|\leq |\alpha|+\tau_R$. Hence we may bound \eqref{9e8.01} by
\begin{align} \label{4e5.72}
\E( \vert K_R^{2}-K_R^{3}\vert) &\leq    \frac{1}{V(R)} \E\Big(\sum_{\substack{\alpha: 0\leq |\alpha|\leq [N_R]}} \sum_{\substack{\beta: \beta\geq \alpha,\\ |\beta|\leq |\alpha|+\tau_R}}\sum_{\substack{\gamma: \gamma\geq \alpha, \\|\gamma|\leq |\alpha|+\tau_R}}    1_{Y^\beta-Y^\gamma \in \cN(0)}   \nn\\
&\times \Big[1_{\{  |\beta|-\tau_R < \zeta_\beta^1\leq |\beta|\}} +1_{\{ |\gamma|-\tau_R< \zeta_\gamma^1 \leq |\gamma|\}}\Big]\Big).
\end{align}
Set
\begin{align} \label{4e5.71}
 I_0(R):=\frac{1}{V(R)} \E\Big(\sum_{\substack{\alpha: 0\leq |\alpha|\leq [N_R]}} \sum_{\substack{\beta: \beta\geq \alpha,\\ |\beta|\leq |\alpha|+\tau_R}}\sum_{\substack{\gamma: \gamma\geq \alpha, \\|\gamma|\leq |\alpha|+\tau_R}}    1_{Y^\beta-Y^\gamma \in \cN(0)}   1_{\{  |\beta|-\tau_R < \zeta_\beta^1\leq |\beta|\}}  \Big).
 \end{align}
 By symmetry between $\beta$ and $\gamma$, one can check that \eqref{4e5.72} implies
 \begin{align*} 
\E( \vert K_R^{2}-K_R^{3}\vert) &\leq    2 I_0(R).
\end{align*}
It suffices to show that  $I_0^R \to 0$ as $R\to \infty$. \\

Recall the definition of $\zeta_\beta^1$ from \eqref{5e11.4}.
In order that $|\beta|-\tau_R < \zeta_\beta^1\leq |\beta|$ occurs on the event $\{Y^\beta \neq \Delta\}$, there has to be some $(|\beta|-\tau_R)^+ < i\leq \vert \beta\vert $ so that 
\begin{align*} 
 Y^{\beta\vert i}  \in S\Big(\sum_{k=0}^{i-1} \widetilde{Z}_k\Big),
\end{align*}
which means that at time $i$, the particle $\beta\vert i$ is sent to a place that has been visited by some particle $\delta$ before time $i$. Hence we may bound $1_{\{  |\beta|-\tau_R < \zeta_\beta^1\leq |\beta|\}}$ by
\begin{align*} 
&\sum_{i=(|\beta|-\tau_R)^++1}^{\vert \beta\vert}   \sum_{\delta} 1_{|\delta|\leq i-1} 1_{Y^{\beta\vert i}=Y^\delta}.
\end{align*}
Apply the above in \eqref{4e5.71} to see
\begin{align*}  
I_0^R\leq \frac{1}{V(R)} \E\Big(\sum_{\substack{\alpha: 0\leq |\alpha|\leq [N_R]}}& \sum_{\substack{\beta: \beta\geq \alpha,\\ |\beta|\leq |\alpha|+\tau_R}}\sum_{\substack{\gamma: \gamma\geq \alpha, \\|\gamma|\leq |\alpha|+\tau_R}}   1_{Y^\beta-Y^\gamma \in \cN(0)} \nn\\
&\sum_{i=(|\beta|-\tau_R)^++1}^{\vert \beta\vert}    \sum_{\delta} 1_{|\delta|\leq i-1} 1_{Y^{\beta\vert i}=Y^\delta}\Big).
\end{align*}
Let $I_1^{R}$ and $I_2^{R}$ be respectively the sum of $i\leq |\alpha|$ and $i>|\alpha|$ on the right-hand side term above, that is, we define
\begin{align} \label{5e4.3}
I_1^{R}:= \frac{1}{V(R)} \E\Big(\sum_{\substack{\alpha: 0\leq |\alpha|\leq [N_R]}} &\sum_{\substack{\beta: \beta\geq \alpha,\\ |\beta|\leq |\alpha|+\tau_R}}\sum_{\substack{\gamma: \gamma\geq \alpha, \\|\gamma|\leq |\alpha|+\tau_R}}   1_{Y^\beta-Y^\gamma \in \cN(0)} \nn\\
&\sum_{i=(|\beta|-\tau_R)^++1}^{\vert \alpha\vert}    \sum_{\delta} 1_{|\delta|\leq i-1} 1_{Y^{\beta\vert i}=Y^\delta}\Big)
\end{align}
and
\begin{align} \label{5e4.5}
I_2^{R}:= \frac{1}{V(R)} \E\Big(\sum_{\substack{\alpha: 0\leq |\alpha|\leq [N_R]}} &\sum_{\substack{\beta: \beta\geq \alpha,\\ |\beta|\leq |\alpha|+\tau_R}}\sum_{\substack{\gamma: \gamma\geq \alpha, \\|\gamma|\leq |\alpha|+\tau_R}}   1_{Y^\beta-Y^\gamma \in \cN(0)}\nn\\
& \sum_{i=|\alpha|+1}^{\vert \beta\vert}    \sum_{\delta} 1_{|\delta|\leq i-1} 1_{Y^{\beta\vert i}=Y^\delta}\Big).
\end{align}
Then $I_0^R\leq I_1^R+I_2^R$ and it suffices to show that 
\begin{align*} 
\lim_{R\to \infty} I_1^{R}=0 \quad \text{ and } \quad\lim_{R\to \infty} I_2^{R}=0.
\end{align*}

\subsection{Convergence of $I_1^R$}

To calculate $I_1^R$, we set $|\alpha|=k$ for some $0\leq k\leq [N_R]$. Let $|\beta|=k+l$ and $|\gamma|=k+m$ for some $0\leq l,m\leq \tau_R$. Noticing that $i\leq |\alpha|$, we get $\beta|i=\alpha|i$. Next, since $|\delta|\leq i-1\leq |\alpha|-1$, we must have $\delta$ branches off the family tree of $\alpha,\beta,\gamma$ before time $|\alpha|=k$. For any $(k+l-\tau_R)^++1\leq i\leq k$,  set $|\delta\wedge \alpha|=j$ for some $0\leq j\leq i-1$ and let $|\delta|=j+n$ for some $0\leq n\leq i-1-j$. The above case is similar to $J_1^{(i)}$ as in Figure \ref{fig1} except for the notation.
Now  write the sum in $I_1^R$ as
\begin{align*} 
I_1^R= \frac{1}{V(R)} \E\Big( &\sum_{k=0}^{[N_R]}\sum_{l=0}^{\tau_R}\sum_{m=0}^{\tau_R}  \sum_{i=(k+l-\tau_R)^++1}^{k}   \sum_{j=0}^{i-1} \sum_{n=0}^{i-1-j} \sum_{\alpha:  |\alpha|=k}  \sum_{\substack{\beta: \beta\geq \alpha,\\ |\beta|=k+l}}\sum_{\substack{\gamma: \gamma\geq \alpha, \\|\gamma|=k+m}} \sum_{\substack{\delta: \delta \geq \alpha|j\\ |\alpha|=j+n}}   \nn\\
&1_{\{Y^\beta, Y^\gamma, Y^\delta \neq \Delta\} } 1_{Y^\beta-Y^\gamma \in \cN(0)}  1_{Y^{\alpha\vert i}=Y^\delta}\Big).
\end{align*}
 
Recall from \eqref{5e7.31} to see 
\[
Y^\beta-Y^\gamma=\sum_{t=k+1}^{k+m}  W^{\beta|t}+\sum_{s=k+1}^{k+l}  W^{\gamma|s}.
\]
The above is independent of $\cH_{\alpha}\vee \cH_\delta$. Hence by conditioning on $\cH_\alpha \vee\cH_\delta$, on the event $\{Y^\beta, Y^\gamma, Y^\delta \neq \Delta\}$ we get
\begin{align*} 
\P({Y^\beta-Y^\gamma \in \cN(0)}|\cH_\alpha \vee \cH_\delta)\leq \frac{C}{(1+l+m)^{d/2}},
\end{align*}
where the last inequality is by Lemma \ref{9l4.2}.

Next, notice that $\alpha\wedge \delta=\alpha|j$ with $|\delta|=j+n$. Use \eqref{5e7.31} again to get
\begin{align} \label{e1.33}
Y^{\alpha|i}-Y^\delta=\sum_{t=j+1}^{i} W^{\alpha|t}+\sum_{s=j+1}^{j+n} W^{\delta|s}.
\end{align}
Similar to \eqref{4e1.23}, we let $W^{\delta|(j+n)}=e_i$ for some $1\leq i\leq V(R)$ and use \eqref{e1.33} to  obtain
\begin{align} \label{e1.34}
   \P(Y^{\alpha|i}=Y^\delta)  & =  \frac{1}{V(R)} \P\Big(\sum_{t=j+1}^{i} W^{\alpha|t}+\sum_{s=j+1}^{j+n-1} W^{\delta|s}\in \cN(0)\Big)\nn \\
    &\leq \frac{1}{V(R)}\frac{C}{(1+(i-1-j)+n)^{d/2}}.
\end{align}

Now we are left with
\begin{align*} 
I_1^R\leq \frac{1}{V(R)} &\sum_{k=0}^{[N_R]} \sum_{l=0}^{\tau_R}\sum_{m=0}^{\tau_R} \sum_{i=(k+l-\tau_R)^++1}^{k}   \sum_{j=0}^{i-1} \sum_{n=0}^{i-1-j}\sum_{\alpha:  |\alpha|=k}\sum_{\substack{\delta: \delta \geq \alpha|j\\ |\alpha|=j+n}}\sum_{\substack{\beta: \beta\geq \alpha,\\ |\beta|=k+l}}\sum_{\substack{\gamma: \gamma\geq \alpha, \\|\gamma|=k+m}}   \nn\\
&   \P({Y^\beta, Y^\gamma, Y^\delta\neq \Delta}) \frac{C}{(1+l+m)^{d/2}}\frac{1}{V(R)}\frac{C}{(1+(i-1-j)+n)^{d/2}}.
\end{align*}
The probability of $\{Y^\beta, Y^\gamma, Y^\delta\neq \Delta\}$ gives $p(R)^k p(R)^l p(R)^m p(R)^n$ while the sum of $\alpha, \beta,\gamma,\delta$ gives $V(R)^k V(R)^l V(R)^m V(R)^n$. So the above is at most
\begin{align*} 
I_1^R\leq \frac{C}{V(R)^2} & \sum_{k=0}^{[N_R]}  \sum_{l=0}^{\tau_R}\sum_{m=0}^{\tau_R}  \sum_{i=(k+l-\tau_R)^++1}^{k}   \sum_{j=0}^{i-1} \sum_{n=0}^{i-1-j}\nn\\
&  (V(R)p(R))^{k+l+m+n} \frac{1}{(1+l+m)^{d/2}} \frac{1}{(1+(i-1-j)+n)^{d/2}} .
\end{align*}
Use $k+l+m+n\leq 4[N_R]$ and $V(R)p(R)\leq e^{\theta/N_R}$ to get $(V(R)p(R))^{k+l+m+n}\leq e^{4\theta}$. Next, the sum of $n$ gives
\begin{align*} 
\sum_{n=0}^{i-1-j}\frac{1}{(1+(i-1-j)+n)^{d/2}} \leq \frac{C}{(1+(i-1-j))^{d/2-1}}.
\end{align*}
The sum of $j$ gives
\begin{align*} 
  \sum_{j=0}^{i-1} \frac{C}{(1+(i-1-j))^{d/2-1}}\leq CI(i-1)\leq CI(N_R).
\end{align*}
The sum of $i$ is bounded by $\tau_R-l\leq \tau_R$. The sum of $m$ gives at most $C/(1+l)^{d/2-1}$ and the sum of $l$ gives $CI(\tau_R)\leq CI(N_R)$. Finally the sum of $k$ gives $[N_R]+1\leq 2N_R$. Combine the above to conclude
\begin{align} \label{4e5.10}
I_1^R\leq \frac{C}{V(R)^2} e^{4\theta}  \times CI(N_R) \times \tau_R \times CI(N_R) \times 2N_R\leq C\frac{N_R\tau_R}{V(R)^2} I(N_R)^2.
\end{align}
When $d\geq 5$, use $I(N_R)\leq C$ and $N_R=R^d$ to get $I_1^R\leq C\frac{\tau_R}{R^d} \to 0$ by \eqref{5e2.10}. When $d=4$, we get
\begin{align*} 
I_1^R\leq C\frac{\frac{R^4}{\log R} \frac{R^4/\log R}{\log R}}{(R^4)^2} (C\log R)^2\leq  \frac{C}{\log R} \to 0.
\end{align*}

\subsection{Convergence of $I_2^R$}
 
Turning to $I_2^R$, again we let $|\alpha|=k$ for some $0\leq k\leq [N_R]$. Let $|\beta|=k+l$ and $|\gamma|=k+m$ for some $0\leq l,m\leq \tau_R$. Then we write
\begin{align*} 
I_2^R= \frac{1}{V(R)} \E\Big(\sum_{k=0}^{[N_R]} \sum_{l=0}^{\tau_R} \sum_{m=0}^{\tau_R} \sum_{i=k+1}^{k+l} \sum_{\alpha: |\alpha|=k}\sum_{\substack{\beta: \beta\geq \alpha,\\ |\beta|=k+l}}\sum_{\substack{\gamma: \gamma\geq \alpha, \\|\gamma|=k+m}}   1_{Y^\beta-Y^\gamma \in \cN(0)}    \sum_{\delta} 1_{|\delta|\leq i-1} 1_{Y^{\beta\vert i}=Y^\delta}\Big).
\end{align*}

Since $|\delta|\leq i-1<|\beta|$, there are two cases for the generation when $\delta$ branches off the family tree of $\beta,\gamma$: 
\begin{align*} 
\text{(1)}  \ \vert \gamma \wedge \delta\vert \leq \vert \beta \wedge \delta\vert ; \quad \text{(2)}  \   \vert \gamma \wedge \delta\vert >\vert \beta \wedge \delta\vert .
\end{align*}
Let $I_2^{(1, R)}$ (resp. $I_2^{(2, R)}$) denote the contribution to $I_2^R$ from case (1) (resp. case (2)). Roughly speaking, case (1) gives that $\delta$ branches off the family tree of $\beta,\gamma$ through the $\beta$ line while case (2) is through the $\gamma$ line after $\alpha=\beta \wedge \gamma$. See Figure \ref{fig2} below.  \\ 

\begin{figure}[ht]
  \begin{center}
    \includegraphics[width=0.75 \textwidth]{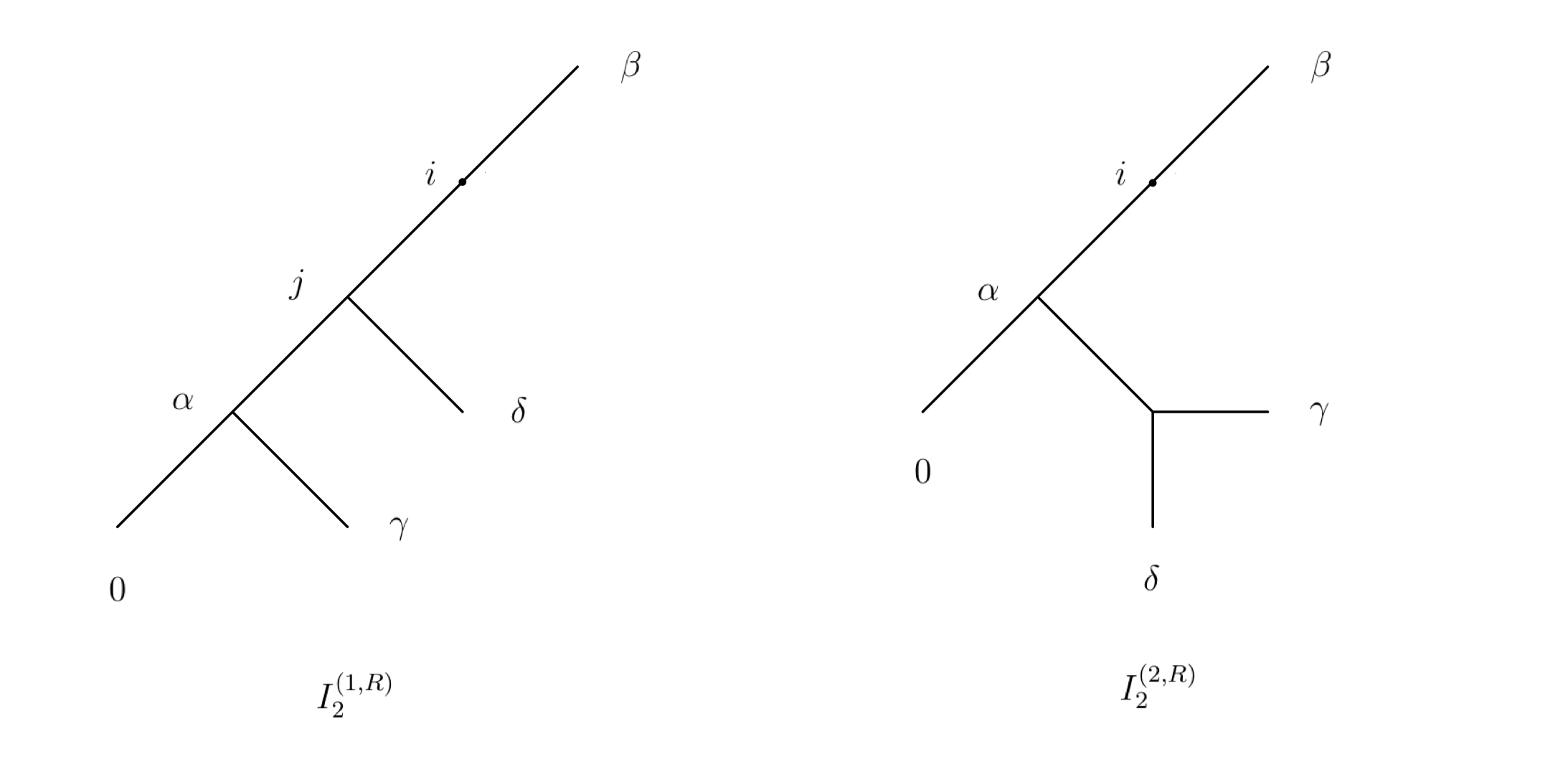}
    \caption[Branching Particle System]{\label{fig2}   Two cases for $I_2^R$.
      }
  \end{center}
\end{figure}

\no {\bf (1) Case\ for\ $I_2^{(1,R)}$}. Since $|\delta|\leq i-1$, we let $\beta \wedge \delta=\beta|j$ for some $0\leq j\leq i-1$. Set $|\delta|=j+n$ for some $0\leq n\leq i-1-j$. Now write the sum in $I_2^{(1,R)}$ as
\begin{align*} 
I_2^{(1,R)}= \frac{1}{V(R)} \E\Big(&\sum_{k=0}^{[N_R]} \sum_{l=0}^{\tau_R} \sum_{m=0}^{\tau_R} \sum_{i=k+1}^{k+l}   \sum_{j=0}^{i-1} \sum_{n=0}^{i-1-j} \sum_{\alpha: |\alpha|=k}\sum_{\substack{\beta: \beta\geq \alpha,\\ |\beta|=k+l}}\sum_{\substack{\gamma: \gamma\geq \alpha, \\|\gamma|=k+m}} \sum_{\substack{\delta: \delta\geq \beta|j, \\|\delta|=j+n} } \nn\\
&1_{\{Y^\beta, Y^\gamma, Y^\delta \neq \Delta\} } 1_{Y^\beta-Y^\gamma \in \cN(0)}  1_{Y^{\beta\vert i}=Y^\delta}\Big).
\end{align*}
Recall from \eqref{5e7.31} to see 
 \begin{align} \label{e4.11}
(Y^\beta-Y^{\beta|i})-(Y^\gamma-Y^{\gamma|(k+1)})=\sum_{t=i+1}^{k+l}  W^{\beta|t}+\sum_{s=k+2}^{k+m}  W^{\gamma|s},
\end{align}
The above are independent of $\cH_{\beta|i}\vee \cH_\delta$. Hence by conditioning on $\cH_{\beta|i}\vee \cH_\delta$, we use \eqref{e4.11} to get on the event $\{Y^\beta, Y^\gamma, Y^\delta \neq \Delta\}$,
\begin{align*} 
\P({Y^\beta-Y^\gamma \in \cN(0)}|\cH_{\beta|i} \vee \cH_\delta)& \leq \sup_x\P\Big({\sum_{t=i+1}^{k+l}  W^{\beta|t}+\sum_{s=k+2}^{k+m}  W^{\gamma|s}+x \in \cN(0)}\Big) \nn\\
&\leq \frac{C}{(1+(k+l-i)+m)^{d/2}}.
\end{align*}
where the last inequality is by Lemma \ref{9l4.2}.

Next, notice that $\beta\wedge \delta=\beta|j$ with $|\delta|=j+n$. Use \eqref{5e7.31} again to get
 \begin{align} \label{e4.14}
Y^{\beta|i}-Y^\delta=\sum_{t=j+1}^{i} W^{\beta|t}+\sum_{s=j+1}^{j+n} W^{\delta|s}.
\end{align}
Similar to \eqref{e1.34}, we use \eqref{e4.14} to obtain
 \begin{align*} 
    \P(Y^{\beta|i}=Y^\delta)& = \frac{1}{V(R)} \P\Big(\sum_{t=j+1}^{i} W^{\beta|t}+\sum_{s=j+1}^{j+n} W^{\delta|s}\in \cN(0)\Big) \\
    &\leq\frac{1}{V(R)}\frac{C}{(1+(i-1-j)+n)^{d/2}}.
\end{align*}
 
Now we are left with
\begin{align*} 
I_2^{(1,R)}&\leq \frac{1}{V(R)} \sum_{k=0}^{[N_R]} \sum_{l=0}^{\tau_R}\sum_{m=0}^{\tau_R}  \sum_{i=k+1}^{k+l}   \sum_{j=0}^{i-1} \sum_{n=0}^{i-1-j} \sum_{\alpha:  |\alpha|=k}\sum_{\substack{\beta: \beta\geq \alpha,\\ |\beta|=k+l}}\sum_{\substack{\gamma: \gamma\geq \alpha, \\|\gamma|=k+m}} \sum_{\substack{\delta: \delta\geq \beta|j, \\|\delta|=j+n} }   \\
&\P({Y^\beta, Y^\gamma, Y^\delta\neq \Delta})\frac{C}{(1+(l+k-i)+m)^{d/2}}\frac{1}{V(R)}\frac{C}{(1+(i-1-j)+n)^{d/2}}.\nn
\end{align*}
The probability $\P({Y^\beta, Y^\gamma, Y^\delta\neq \Delta})$ is bounded by $p(R)^k p(R)^l p(R)^m p(R)^n$ while the sum of $\alpha, \beta,\gamma,\delta$ gives $V(R)^k V(R)^l V(R)^m V(R)^n$. So the above is at most
\begin{align*} 
 \frac{Ce^{4\theta}}{V(R)^2} & \sum_{k=0}^{[N_R]}  \sum_{l=0}^{\tau_R}\sum_{m=0}^{\tau_R}   \sum_{i=k+1}^{k+l}   \sum_{j=0}^{i-1} \sum_{n=0}^{i-1-j}     \frac{1}{(1+(l+k-i)+m)^{d/2}}  \frac{1}{(1+(i-1-j)+n)^{d/2}}.
\end{align*}
where we have used $(V(R)p(R))^{k+l+m+n}\leq e^{4\theta}$ as before. The sum of $n$ gives $C/(1+(i-1-j))^{d/2-1}$ and the sum of $j$ gives
\begin{align*} 
  \sum_{j=0}^{i-1} \frac{C}{(1+(i-1-j))^{d/2-1}}\leq CI(i-1)\leq CI(N_R).
\end{align*}
The sum of $i$ is equal to
\begin{align*} 
  \sum_{i=k+1}^{k+l} \frac{1}{(1+(l+k-i)+m)^{d/2}} \leq \frac{C}{(1+m)^{d/2-1}}.
\end{align*}
Then the sum of $m$ is at most $CI(\tau_R)\leq CI(N_R)$. The sum of $l$ gives $1+\tau_R\leq 2\tau_R$ and the sum of $k$ is equal to $[N_R]+1\leq 2N_R$. Combine the above to see
\begin{align*} 
I_2^{(1,R)}\leq \frac{Ce^{4\theta}}{V(R)^2}  \times CI(N_R) \times CI(N_R) \times 2\tau_R\times 2N_R\leq C\frac{N_R\tau_R}{V(R)^2} I(N_R)^2.
\end{align*}
The right-hand side above is identical to that in \eqref{4e5.10} and so $I_2^{(1,R)}\to 0$ as $R\to \infty$.\\

\no {\bf (2) Case\ for\ $I_2^{(2,R)}$}. Now that $|\delta\wedge \gamma|>|\delta\wedge \beta|$, we must have $\delta$ branches off the family tree of $\beta, \gamma$ from $\gamma$ after $\gamma \wedge \beta$. That is, we let $\delta\wedge \gamma=\gamma|(k+j)$ for some $1\leq j\leq m\wedge (i-1-k)$ and $|\delta|=k+j+n$ for some $0\leq n\leq i-1-k-j$. See Figure \ref{fig2} for illustration. Now  write the sum in $I_2^{(2,R)}$ as
\begin{align*} 
I_2^{(2,R)}= \frac{1}{V(R)} \E\Big(&\sum_{k=0}^{[N_R]}\sum_{l=0}^{\tau_R} \sum_{m=0}^{\tau_R} \sum_{i=k+1}^{k+l}   \sum_{j=1}^{m\wedge (i-1-k)} \sum_{n=0}^{i-1-k-j}  \sum_{\alpha: |\alpha|=k} \sum_{\substack{\beta: \beta\geq \alpha,\\ |\beta|=k+l}}\sum_{\substack{\gamma: \gamma\geq \alpha, \\|\gamma|=k+m}}    \sum_{\substack{\delta: \delta\geq \gamma|(k+j), \\|\delta|=k+j+n} } \nn\\
&1_{\{Y^\beta, Y^\gamma, Y^\delta \neq \Delta\} } 1_{Y^\beta-Y^\gamma \in \cN(0)}  1_{Y^{\beta\vert i}=Y^\delta}\Big).
 \end{align*}

 Recall from \eqref{5e7.31} to see 
\[
Y^\beta-Y^{\beta|i}-(Y^\gamma-Y^{\gamma|(k+j+1)})=\sum_{t=i+1}^{k+l}  W^{\beta|t}+\sum_{s=k+j+2}^{k+m}  W^{\gamma|s},
\]
The above are independent of $\cH_{\beta|i}\vee \cH_\delta$. Hence by conditioning on $\cH_{\beta|i} \vee\cH_\delta$, on the event $\{Y^\beta, Y^\gamma, Y^\delta \neq \Delta\}$ we get
\begin{align*} 
\P({Y^\beta-Y^\gamma \in \cN(0)}|\cH_{\beta|i} \vee \cH_\delta)& \leq \sup_x\P\Big({\sum_{t=i+1}^{k+l}  W^{\beta|t}+\sum_{s=k+j+2}^{k+m}  W^{\gamma|s}+x \in \cN(0)}\Big) \nn\\
&\leq \frac{C}{(1+(l+k-i)+(m-j))^{d/2}}.
\end{align*}
where the last inequality is by Lemma \ref{9l4.2}.

Next, notice that $\beta\wedge \delta=\beta|j$ with $|\delta|=j+n$. Use \eqref{5e7.31} again to get
\begin{align} \label{e2.46}
Y^{\beta|i}-Y^\delta=\sum_{t=k+1}^{i} W^{\beta|t}+\sum_{s=k+1}^{k+j+n} W^{\delta|s}.
\end{align}
Similar to \eqref{e1.34}, we use \eqref{e2.46} to obtain
 \begin{align*} 
    &\P(Y^{\beta|i}=Y^\delta)= \frac{1}{V(R)} \P\Big(\sum_{t=k+1}^{i} W^{\beta|t}+\sum_{s=k+1}^{k+j+n} W^{\delta|s}\in \cN(0)\Big) \\
    &\leq\frac{1}{V(R)}\frac{C}{(1+(i-1-k)+j+n)^{d/2}}\leq \frac{1}{V(R)}\frac{C}{(1+j+n)^{d/2}}.
\end{align*}

Now we are left with
\begin{align*} 
I_2^{(2,R)}\leq &\frac{1}{V(R)} \sum_{k=0}^{[N_R]}\sum_{l=0}^{\tau_R} \sum_{m=0}^{\tau_R} \sum_{i=k+1}^{k+l}   \sum_{j=1}^{m} \sum_{n=0}^{i-1-k-j}  \sum_{\alpha: |\alpha|=k} \sum_{\substack{\beta: \beta\geq \alpha,\\ |\beta|=k+l}}\sum_{\substack{\gamma: \gamma\geq \alpha, \\|\gamma|=k+m}}    \sum_{\substack{\delta: \delta\geq \gamma|(k+j), \\|\delta|=k+j+n} } \nn\\
&\P({Y^\beta, Y^\gamma, Y^\delta\neq \Delta}) \frac{C}{(1+(l+k-i)+(m-j))^{d/2}} \frac{1}{V(R)}\frac{C}{(1+j+n)^{d/2}},
\end{align*}
where we have replaced $j\leq m\wedge (i-1-k)$ by $j\leq m$.
The probability $\P({Y^\beta, Y^\gamma, Y^\delta\neq \Delta}) $ is bounded by $p(R)^k p(R)^l p(R)^m p(R)^n$ while the sum of $\alpha, \beta,\gamma,\delta$ gives $V(R)^k V(R)^l V(R)^m V(R)^n$. So the above is at most
\begin{align*} 
I_2^{(2,R)}\leq \frac{Ce^{4\theta}}{V(R)^2} & \sum_{k=0}^{[N_R]}  \sum_{l=0}^{\tau_R}\sum_{m=0}^{\tau_R}   \sum_{i=k+1}^{k+l}   \sum_{j=1}^{m } \sum_{n=0}^{i-1-k-j} \nn\\
& \frac{1}{(1+(l+k-i)+(m-j))^{d/2}} \frac{1}{(1+j+n)^{d/2}}.
\end{align*}
where we have used $(V(R)p(R))^{k+l+m+n}\leq e^{4\theta}$ as before. The sum of $n$ gives $C/(1+j)^{d/2-1}$ and the sum of $i$ gives
\begin{align*} 
  \sum_{i=k+1}^{k+l} \frac{1}{(1+(l+k-i)^++(m-j))^{d/2}} \leq \frac{C}{(1+(m-j))^{d/2-1}}.
\end{align*}
We are arriving at
\begin{align*} 
I_2^{(2,R)}\leq \frac{Ce^{4\theta}}{V(R)^2} & \sum_{k=0}^{[N_R]}  \sum_{l=0}^{\tau_R}\sum_{m=0}^{\tau_R}   \sum_{j=1}^{m }   \frac{1}{(1+(m-j))^{d/2-1}} \frac{1}{(1+j)^{d/2-1}}.
\end{align*}
Interchange the sum of $m,j$ gives
\begin{align*} 
 &\sum_{j=1}^{\tau_R} \frac{1}{(1+j)^{d/2-1}} \sum_{m=j}^{\tau_R}   \frac{1}{(1+(m-j))^{d/2-1}}\nn\\
 &\leq  \sum_{j=1}^{\tau_R} \frac{1}{(1+j)^{d/2-1}}  I(\tau_R)\leq I(\tau_R)^2\leq I(N_R)^2.
\end{align*}
Finally, we get
\begin{align*} 
I_2^{(2,R)}& \leq \frac{Ce^{4\theta}}{V(R)^2}   \sum_{k=0}^{[N_R]}  \sum_{l=0}^{\tau_R} I(N_R)^2 \leq C\frac{N_R\tau_R}{V(R)^2} I(N_R)^2.
\end{align*}
The right-hand side above is identical to that in \eqref{4e5.10} and so $I_2^{(2,R)}\to 0$ as $R\to \infty$.\\

\bibliographystyle{plain}
\def\cprime{$'$}

\end{document}